\documentclass[11pt]{amsart}
\usepackage{verbatim,amssymb,amsmath,graphicx,hyperref}

  \scrollmode

\newcommand{\R}{{\mathbb R}}
\newcommand{\N}{{\mathbb N}}

\newcommand{\EE}{{\mathbb E}}
\newcommand{\PP}{{\mathbb P}}

\newcommand{\eul}{{\widehat X}}
\newcommand{\eultr}{{\widehat Z}}
\newcommand{\ind}{1}
\newcommand{\utrn}{\underline {t_r}_n}
\newcommand{\usn}{\underline {s}_n}
\newcommand{\utn}{\underline {t}_n}

\newcommand{\sgn}{\operatorname{sgn}}
\newcommand{\eps}{\varepsilon}

\newcommand{\F}{{\mathcal F}}
\newcommand{\Tm}{{\mathcal T}}
\theoremstyle{plain}
\newtheorem{theorem}{Theorem}
\newtheorem{prop}{Proposition}
\newtheorem{lemma}{Lemma}

\newtheorem{conj}{Conjecture}

\theoremstyle{definition}
\newtheorem{rem}{Remark}
\newtheorem{ex}{Example}

\begin{document}
\title[A method of order $3/4$  for SDEs with discontinuous drift coefficient]{A strong order $3/4$ method for SDEs with discontinuous drift coefficient}

\author[M\"uller-Gronbach]
{Thomas M\"uller-Gronbach}
\address{
Faculty of Computer Science and Mathematics\\
University of Passau\\
Innstrasse 33 \\
94032 Passau\\
Germany} \email{thomas.mueller-gronbach@uni-passau.de}

\author[Yaroslavtseva]
{Larisa Yaroslavtseva}
\address{
Faculty of Computer Science and Mathematics\\
University of Passau\\
Innstrasse 33 \\
94032 Passau\\
Germany} \email{larisa.yaroslavtseva@uni-passau.de}

\begin{abstract}
In this paper we study strong approximation of the solution of a scalar stochastic differential equation (SDE)  at the final time in the case when the drift coefficient  may have  discontinuities in space.
Recently it has been shown in~\cite{MGY18a}  that for scalar  SDEs with a piecewise Lipschitz drift coefficient and a Lipschitz diffusion coefficient that is non-zero at the discontinuity points  of the drift coefficient the classical Euler-Maruyama scheme achieves  an $L_p$-error rate of at least $1/2$ for all $p\in [1,\infty)$. Up to now this was the best $L_p$-error rate available in the literature for equations of that type.  
In the present paper we construct a method based on finitely many evaluations of the driving Brownian motion that  even achieves an $L_p$-error rate of at least $3/4$ for all $p\in [1,\infty)$ under  additional piecewise smoothness assumptions on the coefficients. To obtain this result we   prove in particular that a quasi-Milstein scheme achieves an $L_p$-error rate of at least $3/4$ in the case of coefficients that are both 
Lipschitz continuous and piecewise differentiable with Lipschitz continuous derivatives, which is of interest in itself. 
\end{abstract}

\maketitle

\section{Introduction}

Consider a scalar autonomous stochastic differential equation (SDE)
\begin{equation}\label{sde000}
\begin{aligned}
dX_t & = \mu(X_t) \, dt + \sigma(X_t) \, dW_t, \quad t\in [0,1],\\
X_0 & = x_0
\end{aligned}
\end{equation}
with deterministic initial value $x_0\in\R$,  drift coefficient $\mu\colon\R\to\R$,  diffusion coefficient $\sigma\colon \R\to\R$, $1$-dimensional driving Brownian motion $W$ and assume that~\eqref{sde000} has a unique strong solution $X$. In this paper we study $L_p$-approximation of $X_1$ based on finitely many evaluations of $W$ at  points in $[0,1]$ in the case when $\mu$ may have finitely many discontinuity points. 

Numerical   approximation  of SDEs with 
a 
drift coefficient that is discontinuous in space has gained a lot of interest in recent years, see~\cite{ g98b, gk96b} for results on convergence in probability and  almost sure convergence  of the Euler-Maruyama scheme  and~\cite{ GLN17, HalidiasKloeden2008,  LS16,  LS15b, LS18, NSS19, Tag16, Tag2017b, Tag2017a} for results on $L_p$-approximation. Up to now the most far going results 
on
$L_p$-approximation have been achieved under the following two 
assumptions
on the coefficients $\mu$ and $\sigma$.

\begin{itemize}
\item[(A1)] There exist $k\in\N$ 
and $\xi_0, \ldots, \xi_{k+1}\in [-\infty,\infty]$ with $-\infty=\xi_0<\xi_1<\ldots < \xi_k <\xi_{k+1}=\infty$ such that
 $\mu$ is Lipschitz continuous on the interval $(\xi_{i-1}, \xi_i)$ for all $i\in\{1, \ldots, k+1\}$,
\item[(A2)] $\sigma$ is Lipschitz continuous on $\R$ and $\sigma(\xi_i) \neq 0$ for all $i\in\{1,\ldots,k\}$.
\end{itemize}
Note that under the assumptions
(A1) and (A2) the equation~\eqref{sde000} has a unique strong solution, see ~\cite[Theorem 2.2]{LS16}. 
In~\cite{LS16, LS15b} a numerical 
method
has been  
constructed which is based on a suitable transformation of the solution $X$ and which achieves,  under the assumptions (A1) and (A2), an $L_2$-error rate of at least $1/2$ in terms of the number of evaluations of $W$.
In~\cite{LS18} it
has been
shown that the Euler-Maruyama
 scheme achieves an $L_2$-error rate of  at least $1/4-$
in terms of the number of evaluations of $W$  
if (A1) and (A2) are satisfied and, additionally, the coefficients $\mu$ and $\sigma$ are bounded. 
 In~\cite{NSS19} an adaptive Euler-Maruyama scheme
has been
constructed, which achieves, under the assumptions (A1) and (A2), 
an $L_2$-error rate of at least $1/2-$ in terms of the average number of evaluations of $W$. Finally, in~\cite{MGY18a} it
has been 
shown that, 
under the assumptions (A1) and (A2), the Euler-Maruyama 
scheme 
 in fact
achieves for all $p\in [1,\infty)$ an $L_p$-error rate of at least $1/2$  
in terms of the number of evaluations of $W$
as in the case of SDEs with 
globally Lipschitz continuous coefficients.

It is well known that if the coefficients $\mu$ and $\sigma$ are differentiable and have  bounded and Lipschitz continuous derivatives, then the Milstein scheme achieves  
      for all $p\in[1, \infty)$ an $L_p$-error rate of at least $1$ 
in terms of the number of evaluations of $W$,
see e.g.~\cite{HMGR01}. 
 It is therefore natural to ask whether an $L_p$-error rate better than $1/2$ 
can  be achieved by a method based on finitely many evaluations of $W$ also in the case of coefficients $\mu$ and $\sigma$ that satisfy (A1) and (A2) and have additional
 piecewise smoothness properties. To the best of our knowledge the answer to this question
was not known  in the literature up to now. In the present paper we answer this question in the positive. More precisely, we show that if the coefficients $\mu$ and $\sigma$ satisfy (A1) and (A2) and, additionally, the assumption
 \begin{itemize}
\item[(A3)] $\mu$ and $\sigma$ are differentiable on the interval  $(\xi_{i-1}, \xi_i)$ with Lipschitz continuous derivatives  for all $i\in\{1, \ldots, k+1\}$
\end{itemize}
then  an $L_p$-error rate of at least $3/4$    for all $p\in[1, \infty)$
can  be achieved by a method based on evaluations of $W$ at a uniform grid.
More formally, we have the following result, which is an immediate consequence of Theorem~\ref{Thm2} in Section~\ref{threefour}.

\begin{theorem}\label{introthm1}
Assume that $\mu$ and $\sigma$ satisfy (A1) to (A3). Then there exists a sequence of measurable functions $\varphi_n\colon \R^n\to \R$, $n\in\N$, such that for all $p\in [1,\infty)$ there exists $c\in (0,\infty)$ such that for all $n\in\N$,
\begin{equation}\label{introe1}
\EE\bigl[|X_1-\varphi_n(W_{1/n},W_{2/n},\dots,W_1)|^p\bigr]^{1/p}\le c/n^{3/4}.
\end{equation}
\end{theorem}

We
illustrate the statement of Theorem~\ref{introthm1}  by 
the SDE 
\begin{equation}\label{exx1}
\begin{aligned}
dX_t & = (1+X_t)\, 1_{[0,\infty)}(X_t)\, dt + dW_t,\quad t\in [0,1],\\
X_0 & = x_0.
\end{aligned}
\end{equation}
Clearly,  the assumptions (A1) to (A3) are satisfied 
with $k=1$ and $\xi_1=0$. For the SDE~\eqref{exx1}, the strongest result on $L_p$-approximation of $X_1$ which was  available in the literature so far is provided by~\cite[Theorem 1]{MGY18a}, which states that the  Euler-Maruyama scheme achieves an $L_p$-error rate of at least $1/2$  for all $p\in [1,\infty)$. However, by Theorem~\ref{introthm1} we see that  for this SDE  in fact an $L_p$-error rate of at least $3/4$ for all $p\in [1,\infty)$ can be achieved  by a method based on finitely many evaluations of $W$.

We believe that  the upper error
bound \eqref{introe1}
in Theorem~\ref{introthm1} can not be improved  in general by a method  based on $n$  evaluations of  $W$.  See also Conjecture \ref{Conj2} in Section \ref{disc}. 
We  furthermore believe that an $L_p$-error rate better than $3/4$  can not be achieved in general even then when  the coefficients $\mu$ and $\sigma$ satisfy (A1) and (A2) and are, additionally, infinitely often differentiable on the interval  $(\xi_{i-1}, \xi_i)$ with Lipschitz continuous derivatives of all orders  for all $i\in\{1, \ldots, k+1\}$.   A study of these conjectures
will be the subject of future work.

Similarly
 to the approach taken in~\cite{LS16,LS15b}, the proof of Theorem~\ref{introthm1} is based on applying a suitable bi-Lipschitz mapping $G\colon\R\to\R$ to the solution $X$ of~\eqref{sde000}.
Under the assumptions (A1) to (A3) it is possible to construct $G$ in such a way that the transformed solution $G\circ X = (G(X_t))_{t\in [0,1]}$ is the unique strong solution of a new SDE with coefficients that are both globally Lipschitz continuous and piecewise differentiable with Lipschitz continuous derivatives. For the latter SDE we introduce a quasi-Milstein scheme  $(\widehat X_{n,\ell/n})_{\ell =0,\dots,n}$ and prove that $ \widehat X_{n,1}$ achieves  
for all $p\in [1,\infty)$
an $L_p$-error rate of at least $3/4$
in terms of the number of evaluations of $W$ 
for approximating $G(X_1)$. Using the Lipschitz continuity of $G^{-1}$ yields the statement of Theorem~\ref{introthm1} with 
$\varphi_n(W_{1/n},W_{2/n},\dots,W_1)= G^{-1} (\widehat X_{n,1})$.  

To be more precise we introduce the following three assumptions on the coefficients $\mu$ and $\sigma$ of the SDE~\eqref{sde000}, which are stronger than the assumptions (A1) to (A3).

\begin{itemize}
\item[(B1)] $\mu$ and $\sigma$ are Lipschitz continuous on $\R$,
\item[(B2)] there exist $k_\mu,k_\sigma\in\N_0$ and $\xi_0, \ldots, \xi_{k_\mu+1}\in [-\infty,\infty]$  with $-\infty=\xi_0<\xi_1<\ldots < \xi_{k_\mu} <\xi_{k_\mu+1}=\infty$ as well as  $\eta_0, \ldots, \eta_{k_\sigma+1}\in [-\infty,\infty]$  with $-\infty=\eta_0<\eta_1<\ldots < \eta_{k_\sigma} <\eta_{k_\sigma+1}=\infty$ such that
 $\mu$ is differentiable on the interval $(\xi_{i-1}, \xi_i)$ with Lipschitz continuous derivative for all $i\in\{1, \ldots, k_\mu+1\}$ and $\sigma$ is differentiable on the interval $(\eta_{i-1}, \eta_i)$ with Lipschitz continuous derivative for all $i\in\{1, \ldots, k_\sigma+1\}$,
\item[(B3)] $\sigma(\xi_i) \neq 0$ for all $i\in\{1,\ldots,k_\mu\}$ and $\sigma(\eta_i) \neq 0$ for all $i\in\{1,\ldots,k_\sigma\}$.
\end{itemize}

Furthermore, for all $n\in\N$ we define the quasi-Milstein scheme  $(\widehat X_{n,\ell/n})_{\ell =0,\dots,n}$ with step-size $1/n$ associated to the SDE~\eqref{sde000} by $\widehat X_{n,0}= x_0$ and 
\begin{equation*}
\begin{aligned}
\widehat X_{n,(\ell+1)/n}& =\widehat X_{n,\ell/n}+\mu(\widehat X_{n,\ell/n})\cdot 1/n + \sigma(\widehat X_{n,\ell/n})\cdot (W_{(\ell+1)/n}-W_{\ell/n})\\
& \qquad\qquad + \frac{1}{2} \sigma\delta_\sigma(\widehat X_{n,\ell/n})\cdot ((W_{(\ell+1)/n}-W_{\ell/n})^2 - 1/n)
\end{aligned}
\end{equation*}
for $\ell = 0,\dots,n-1$, where $\delta_\sigma(x)=\sigma'(x)$ if $\sigma$ is differentiable at $x$ and $\delta_\sigma(x)=0$ otherwise.  

We then have the following result, which is an immediate consequence of Theorem~\ref{Thm1} in Section~\ref{QM}.

\begin{theorem}\label{introthm2}
Assume that $\mu$ and $\sigma$ satisfy (B1) to (B3) and let $p\in[1,\infty)$. Then there exists $c\in (0,\infty)$ such that for all $n\in\N$,
\begin{equation}\label{jjj}
\EE\bigl[|X_1-\widehat X_{n,1}|^p\bigr]^{1/p}\le c/n^{3/4}.
\end{equation}
If, additionally, $k_\sigma=0$  then there exists $c\in (0,\infty)$ such that for all $n\in\N$,
\begin{equation}\label{ppp}
\EE\bigl[|X_1-\widehat X_{n,1}|^p\bigr]^{1/p}\le c/n.
\end{equation}
\end{theorem}

Note that if $k_\sigma=0$ then $\sigma$ is differentiable on $\R$ and thus $\eul_n$ coincides with the classical  Milstein scheme. 
However, 
the upper error bound \eqref{ppp} was known 
in the literature so far  only in the case of $k_\mu = k_\sigma =0$, see e.g.~\cite{HMGR01}.

For illustration of the statement of Theorem~\ref{introthm2}  we consider the SDEs
\begin{equation}\label{exx2}
\begin{aligned}
dX^{(1)}_t & =  X^{(1)}_t\cdot 1_{[0,\infty)}(X^{(1)}_t)\, dt +  (1+X^{(1)}_t\cdot  1_{[0,\infty)}(X^{(1)}_t))\, dW_t,\quad t\in [0,1],\\
X^{(1)}_0 & = x_0
\end{aligned}
\end{equation}
and
\begin{equation}\label{exx22}
\begin{aligned}
dX^{(2)}_t & =  X^{(2)}_t\cdot 1_{[0,\infty)}(X^{(2)}_t)\, dt +   dW_t,\quad t\in [0,1],\\
X^{(2)}_0 & = x_0,\\
\end{aligned}
\end{equation}
Clearly, the assumptions  (B1) to (B3) are satisfied for the coefficients of the SDE \eqref{exx2} with $k_\mu=k_\sigma=1$ and $\xi_1 =\eta_1= 0$ and for the coefficients of the SDE \eqref{exx22} with $k_\mu=1$, $k_\sigma=0$  and $\xi_1 =0$.
The best possible $L_p$-error rate for approximation of $X^{(1)}_1$ and $X^{(2)}_1$ which was available in the literature so far is equal to $1/2$ and is achieved, e.g., by the Euler-Maruyama scheme.
 However, by Theorem~\ref{introthm2} we see that  for all $p\in [1,\infty)$ the associated quasi-Milstein scheme  achieves 
 an $L_p$-error rate of at least   $3/4$ and  $1$   for approximation of $X^{(1)}_1$ and $X^{(2)}_1$, respectively.

We briefly describe the content of the paper. In Section~\ref{Not} we introduce some notation. Section~\ref{QM} contains our result Theorem~\ref{Thm1} 
on the quasi-Milstein scheme under the assumptions (B1) to (B3). In Section~\ref{threefour} we 
 construct 
the bi-Lipschitz transformation $G$ that is then used to construct a method of order $3/4$ under the assumptions (A1) to (A3). Section~\ref{disc} contains a discussion of our results as well as conjectures with respect to lower error bounds. The proof of Theorem~\ref{Thm1} is carried out in Section~\ref{Proofs}. Section~\ref{Lem} contains 
proofs of two lemmas that are employed in Section~\ref{threefour} for the construction of the mapping $G$.

\section{Notation}\label{Not}
For $A\subset \R$ and a function $f\colon A\to\R$ we put $\|f\|_\infty = \sup_{x\in A}|f(x)|$. For a function $f\colon \R\to\R$ we define $\delta_f\colon \R\to\R$ by
\[
\delta_f(x) = \begin{cases} f'(x), & \text{if $f$ is differentiable in $x$},\\
0, & \text{otherwise.}\end{cases}
\]

\section{A quasi-Milstein scheme for SDEs with Lipschitz continuous coefficients}\label{QM}

Let
$ ( \Omega, \mathcal{F}, \PP ) $ 
be a probability space with a normal filtration
$ ( \mathcal{F}_t )_{ t \in [0,1] } $ and
let
$
  W \colon [0,1] \times \Omega \to \R
$
be an 
$ ( \mathcal{F}_t )_{ t \in [0,1] } $-Brownian motion
on $ ( \Omega, \mathcal{F}, \PP )$. Moreover, let $x_0\in\R$ and let $\mu, \sigma\colon\R\to\R$ be functions that satisfy the following three assumptions. 
\begin{itemize}
\item[(B1)] $\mu$ and $\sigma$ are Lipschitz continuous on $\R$,
\item[(B2)] there exist $k_\mu,k_\sigma\in\N_0$ and $\xi_0, \ldots, \xi_{k_\mu+1}\in [-\infty,\infty]$  with $-\infty=\xi_0<\xi_1<\ldots < \xi_{k_\mu} <\xi_{k_\mu+1}=\infty$ as well as  $\eta_0, \ldots, \eta_{k_\sigma+1}\in [-\infty,\infty]$  with $-\infty=\eta_0<\eta_1<\ldots < \eta_{k_\sigma} <\eta_{k_\sigma+1}=\infty$ such that
 $\mu$ is differentiable on the interval $(\xi_{i-1}, \xi_i)$ with Lipschitz continuous derivative for all $i\in\{1, \ldots, k_\mu+1\}$ and $\sigma$ is differentiable on the interval $(\eta_{i-1}, \eta_i)$ with Lipschitz continuous derivative for all $i\in\{1, \ldots, k_\sigma+1\}$
\item[(B3)] $\sigma(\xi_i) \neq 0$ for all $i\in\{1,\ldots,k_\mu\}$ and $\sigma(\eta_i) \neq 0$ for all $i\in\{1,\ldots,k_\sigma\}$.
\end{itemize}

We consider the SDE
\begin{equation}\label{sde01}
\begin{aligned}
dX_t & = \mu(X_t) \, dt + \sigma(X_t) \, dW_t, \quad t\in [0,1],\\
X_0 & = x_0,
\end{aligned}
\end{equation}
which has a unique strong solution 
due to the assumption  
(B1).

Moreover, for every $p\in (0,\infty)$,
\begin{equation}\label{mom}
\EE\bigl[\|X\|_\infty^p\bigr] < \infty,
\end{equation}
see, e.g.~\cite[Thm. 2.4.4]{Mao08}.

For $n\in\N$ we  
use $\eul_{n}=(\eul_{n,t})_{t\in[0,1]}$ to denote
a time-continuous quasi-Milstein scheme with step-size $1/n$ associated to the SDE \eqref{sde01}, which is defined recursively by
$\eul_{n,0}=x_0$ and
\begin{align*}
\eul_{n,t}=\eul_{n,i/n}&+\mu(\eul_{n,i/n})\cdot (t-i/n)+\sigma(\eul_{n,i/n})\cdot (W_t-W_{i/n})\\
&+\frac{1}{2}\sigma \delta_\sigma (\eul_{n,i/n})\cdot\bigl((W_t-W_{i/n})^2-(t-i/n)\bigr)
\end{align*}
for $t\in (i/n,(i+1)/n]$ 
and $i\in\{0,\ldots,n-1\}$. 
Note that for all $x\not\in\{\eta_1, \ldots, \eta_{k_\sigma}\}$ we have $\delta_\sigma(x) = \sigma'(x)$.

 We have the following error estimates for $\eul_{n}$.
\begin{theorem}\label{Thm1}
Assume (B1) to (B3). Let $p\in [1,\infty)$. Then
there exists $c\in(0, \infty)$ such that for all $n\in\N$, 
\begin{equation}\label{ll3}
\EE\bigl[\|X-\eul_{n}\|_\infty^p\bigr]^{1/p}\leq \frac{c}{ n^{3/4}}. 
\end{equation}
If, additionally, $k_\sigma=0$ then there exists $c\in(0, \infty)$ such that for all $n\in\N$, 
\begin{equation}\label{ll33}
\EE\bigl[\|X-\eul_{n}\|_\infty^p\bigr]^{1/p}\leq \frac{c}{n}. 
\end{equation}
\end{theorem}
The proof of Theorem \ref{Thm1} is postponed to Section~\ref{Proofs}.

\begin{rem}\label{Rem1} 
Note that if (B1) to (B3) are satisfied with $k_\sigma=0$ then $\eul_n$ coincides with the classical time-continuous Milstein scheme. We add that in case of $k_\mu = k_\sigma =0$ and under much stronger smoothness 
assumptions
on $\mu$ and $\sigma$ than stated in (B1) and (B2), the error estimate~\eqref{ll33} is known, see e.g.~\cite[Thm. 10.6.3]{kp92}.
\end{rem}  

\begin{rem}\label{Rem1b}
In~\cite{KruseWu18} a randomized Milstein scheme is constructed that is based on evaluations of $W$ at the grid points $\ell/n$, $\ell=1,\dots,n$, and randomly chosen intermediate points $s_\ell\in ((\ell-1)/n,\ell/n)$, $\ell =1,\dots,n$. This scheme is shown to achieve for all $p\in[1,\infty)$ an $L_p$-error rate of at least $1$ in terms of $n$ under assumptions that are, in comparison with (B1) to (B3), weaker with respect to $\mu$ and stronger with respect to $\sigma$, namely the assumptions that $\mu$ is  Lipschitz continuous on $\R$, $\sigma$ is differentiable on $\R$ with a bounded Lipschitz continuous derivative $\sigma'$ and $\sigma\sigma'$ is Lipschitz continuous on $\R$.
  \end{rem}

\section{A strong order $3/4$ method for SDEs with discontinuous drift coefficient}\label{threefour}
As in Section~\ref{QM} we consider a  probability space
$ ( \Omega, \mathcal{F}, \PP ) $ 
with a normal filtration
$ ( \mathcal{F}_t )_{ t \in [0,1] } $ and we assume that 
$
  W \colon [0,1] \times \Omega \to \R
$
is an 
$ ( \mathcal{F}_t )_{ t \in [0,1] } $-Brownian motion
on $ ( \Omega, \mathcal{F}, \PP )$. In contrast to Section~\ref{QM} we now turn to SDEs with a drift coefficient $\mu$ that may be only piecewise Lipschitz continuous.

Let $x_0\in\R$ and let $\mu, \sigma\colon\R\to\R$ be functions that satisfy the following three assumptions.  
\begin{itemize}
\item[(A1)] There exist $k\in\N$ 
and $\xi_0, \ldots, \xi_{k+1}\in [-\infty,\infty]$ with $-\infty=\xi_0<\xi_1<\ldots < \xi_k <\xi_{k+1}=\infty$ such that
 $\mu$ is Lipschitz continuous on the interval $(\xi_{i-1}, \xi_i)$ for all $i\in\{1, \ldots, k+1\}$,
\item[(A2)] $\sigma$ is Lipschitz continuous on $\R$ and $\sigma(\xi_i) \neq 0$ for all $i\in\{1,\ldots,k\}$,
\item[(A3)] $\mu$ and $\sigma$ are differentiable on the interval  $(\xi_{i-1}, \xi_i)$ with Lipschitz continuous derivatives  for all $i\in\{1, \ldots, k+1\}$.
\end{itemize}

For later purposes we note that (A1) implies the existence of the one-sided limits $\mu(\xi_i-)$ and $\mu(\xi_i+)$ for all $i\in\{1,\dots,k\}$.

We consider the SDE
\begin{equation}\label{sde0}
\begin{aligned}
dX_t & = \mu(X_t) \, dt + \sigma(X_t) \, dW_t, \quad t\in [0,1],\\
X_0 & = x_0,
\end{aligned}
\end{equation}
which has a unique strong solution,
see~\cite[Theorem 2.2]{LS16}.

Our goal is to show that the solution of~\eqref{sde0} at the final time $X_1$ can be approximated in $p$-th mean sense by means of a method based on $W_{1/n},W_{2/n},\dots,W_1$ at least with order $3/4$ in terms of the number $n$ of equidistant evaluations of the driving Brownian motion $W$, see Theorem~\ref{Thm2}. To achieve this goal we adopt 
the
transformation strategy used in~\cite{LS15b} and~\cite{MGY18a}. We show that 
$X_1$ can be obtained by applying a Lipschitz continuous transformation to the solution of an SDE with coefficients satisfying the assumptions (B1) to (B3) in Section~\ref{QM}, and then we employ Theorem~\ref{Thm1}.

We start by introducing the transformation procedure. For $k\in\N$, 
\[
z\in\Tm_k=\{(z_1,\dots,z_k)\in\R^k\colon z_1<\dots<z_k\}
\]
 and $\alpha=(\alpha_1,\dots,\alpha_k)\in\R^k$ we put
\[
\rho_{z,\alpha} =  \begin{cases}
\frac{1}{8 |\alpha_1|}, & \text{if }k=1, \\
\min\bigl(\bigl\{\frac{1}{8 |\alpha_i|}\colon i\in \{1, \ldots, k\}\bigr\} \cup \bigl\{ \frac{z_i-z_{i-1}}{2}\colon i\in \{2, \ldots, k\}\bigr\} \bigr),& \text{if }k\geq 2,
\end{cases}
\]
where we use the convention $1/0 =\infty$. Let $\phi\colon\R\to\R$ be given by
\begin{equation}\label{phi}
\phi(x)=(1-x^2)^4\cdot \ind_{[-1, 1]}(x). 
\end{equation}
For all $k\in\N$, $z\in \Tm_k$, $\alpha\in\R^k$ and $\nu\in (0,\rho_{z,\alpha})$ we define a function $G_{z,\alpha,\nu}\colon\R\to\R$ by
\begin{equation}\label{fct1}
G_{z,\alpha,\nu}(x) = x+\sum_{i=1}^k \alpha_i\cdot (x-z_i)\cdot |x-z_i|\cdot \phi \Bigl(\frac{x-z_i}{\nu}\Bigr).
\end{equation}

The following two technical lemmas provide the properties of the mappings $G_{z,\alpha,\nu}$ that are crucial for our purposes. The proofs of both lemmas are postponed to Section~\ref{Lem}. 

\begin{lemma}\label{lemx1}
Let $k\in\N$, $z\in \Tm_k$, $\alpha\in\R^k$, $\nu\in (0,\rho_{z,\alpha})$ and put $z_0=-\infty$ and $z_{k+1}= \infty$. The function $G_{z,\alpha,\nu}$ has the following properties.
\begin{itemize}
\item[(i)] $G_{z,\alpha,\nu}$  is differentiable on $\R$ with a Lipschitz continuous derivative $G'_{z,\alpha,\nu}$ that satisfies $G'_{z,\alpha,\nu}(z_i) = 1$ for all $i\in\{1,\dots,k\}$ and $\inf_{x\in\R} G_{z,\alpha,\nu}'(x)>0$.
In particular, $G_{z,\alpha,\nu}$ has an inverse $G_{z,\alpha,\nu}^{-1}\colon \R\to \R$ that is Lipschitz continuous. Furthermore, there exists $c \in (0,\infty)$ such that for every $x\in\R$ with $|x|>c$, $G'_{z,\alpha,\nu}(x) = 1$. 
\item[(ii)] For every $i\in\{1,\dots,k+1\}$, the function $G'_{z,\alpha,\nu}$ is two times differentiable on $(z_{i-1},z_i)$ with Lipschitz continuous derivatives $G''_{z,\alpha,\nu}$ and $G'''_{z,\alpha,\nu}$.
\item[(iii)] For every $i\in\{1,\dots,k\}$ the one-sided limits  $G''_{z,\alpha,\nu}(z_i-) $ and $G''_{z,\alpha,\nu}(z_i+)$ exist and satisfy
\[
G''_{z,\alpha,\nu}(z_i-) = -2\alpha_i,\quad G''_{z,\alpha,\nu}(z_i+) = 2\alpha_i.
\]
\end{itemize}
\end{lemma}

\begin{lemma}\label{transform1} Assume (A1) to (A3). Put $\xi=(\xi_1,\dots,\xi_k)$, define $\alpha=(\alpha_1,\dots,\alpha_k)\in \R^k$ by
\[
\alpha_i =\frac{\mu(\xi_i-)-\mu(\xi_i+)}{2 \sigma^2(\xi_i)}
\]
 for  $i\in\{1,\dots,k\}$, and let $\nu\in (0,\rho_{\xi,\alpha})$. Consider the function $G_{\xi,\alpha,\nu}$ and extend $G''_{\xi,\alpha,\nu}\colon \cup_{i=1}^{k+1} (\xi_{i-1},\xi_i)\to \R$ to the whole real line by taking
 \[
 G''_{\xi,\alpha,\nu}(\xi_i) = 2\alpha_i + 2\,\frac{\mu(\xi_i+)-\mu(\xi_i)}{\sigma^2(\xi_i)}  
\]
for $i\in\{1, \ldots, k\}$. Then the functions
\begin{equation}\label{tildecoeff}
\widetilde \mu=(G_{\xi,\alpha,\nu}'\cdot \mu+\tfrac{1}{2}G_{\xi,\alpha,\nu}''\cdot\sigma^2)\circ G_{\xi,\alpha,\nu}^{-1} \, \text{ and }\, \widetilde\sigma=(G_{\xi,\alpha,\nu}'\cdot\sigma)\circ G_{\xi,\alpha,\nu}^{-1}
\end{equation}
satisfy the assumptions (B1) to (B3). 
\end{lemma}

We turn to the transformation of the SDE~\eqref{sde0}. Take $\xi,\alpha,\nu$ as in Lemma~\ref{transform1} and define a stochastic process $Z\colon [0,1]\times \Omega\to \R$ by  
\begin{equation}\label{tr}
Z_t = G_{\xi,\alpha,\nu}(X_t),\quad t\in [0,1].
\end{equation}

\begin{lemma}\label{Ito} 
Assume  (A1) to (A3). Then the process $Z$ given by~\eqref{tr} is the unique strong solution of the SDE
\begin{equation}\label{sde1}
\begin{aligned}
dZ_t & = \widetilde\mu(Z_t) \, dt + \widetilde\sigma(Z_t) \, dW_t, \quad t\in [0,1],\\
Z_0 & = G_{\xi,\alpha,\nu}(x_0)
\end{aligned}
\end{equation}
with $\widetilde \mu$ and $\widetilde \sigma$ given by~\eqref{tildecoeff}.
\end{lemma}
\begin{proof}
Lemma~\ref{lemx1}(i) implies that $G'_{\xi,\alpha,\nu}$ is absolutely continuous. We  therefore may apply  It\^{o}'s lemma
 with $G_{\xi,\alpha,\nu}$ to obtain that $Z$ is a solution of~\eqref{sde1}. According to Lemma~\ref{transform1}, $\widetilde \mu$ and $\widetilde \sigma$ are Lipschitz continuous,
which implies that the solution of~\eqref{sde1} is unique.
\end{proof}

\begin{rem}  The construction of the transformations $G_{z,\alpha,\nu}$ used here   is similar to the construction of the transformations used in~\cite{LS15b} and~\cite{MGY18a}. In the latter works the transformations are also given by \eqref{fct1}, but with $\phi\colon\R\to\R$ defined by
\begin{equation}\label{phi2}
\phi(x)=(1-x^2)^3\cdot \ind_{[-1, 1]}(x)
\end{equation}
in place of~\eqref{phi}.
Note that using~\eqref{phi2} in place of~\eqref{phi}, the functions $G_{z,\alpha,\nu}''$ may  not be differentiable at the points $z_i\pm\nu$ for $i\in\{1, \ldots, k\}$, and therefore  $\widetilde\mu$ may not be differentiable at the points $\xi_i\pm\nu$ for $i\in\{1, \ldots, k\}$. 

\end{rem}

For every $n\in\N$ we use $\eultr_{n}=(\eultr_{n,t})_{t\in[0,1]}$ to denote the 
    time-continuous
quasi-Milstein  scheme  with step-size $1/n$ associated to the SDE \eqref{sde1},  
   see Section~\ref{QM}.
Thus, 
$\eultr_{n,0}=G_{\xi,\alpha,\nu}(x_0)$ and 
\begin{align*}
\eultr_{n,t}=\eultr_{n,i/n}&+\widetilde \mu(\eultr_{n,i/n})\cdot (t-i/n)+\widetilde\sigma(\eultr_{n,i/n})\cdot (W_t-W_{i/n})\\
&+\frac{1}{2}\widetilde\sigma\,\delta_{\widetilde \sigma} (\eultr_{n,i/n})\cdot\bigl((W_t-W_{i/n})^2-(t-i/n)\bigr)
\end{align*}
for $t\in (i/n,(i+1)/n]$ 
and $i\in\{0,\ldots,n-1\}$.

 We have the following error estimates for $G_{\xi,\alpha,\nu}^{-1}\circ \eultr_{n}= (G_{\xi,\alpha,\nu}^{-1}(\eultr_{n,t}))_{t\in[0,1]}$.

\begin{theorem}\label{Thm2} 
Assume (A1) to (A3) and let $p\in [1,\infty)$. Then
there exists $c\in(0, \infty)$ such that for all $n\in\N$, 
\begin{equation}\label{l3}
\EE\bigl[\|X-G_{\xi,\alpha,\nu}^{-1}\circ \eultr_{n}\|_\infty^p\bigr]^{1/p}\leq \frac{c}{ n^{3/4}}. 
\end{equation}
\end{theorem}
\begin{proof}
Using the Lipschitz continuity of $G_{\xi,\alpha,\nu}^{-1}$, 
see Lemma~\ref{transform1}(i),
 the fact that $\widetilde \mu$ and $\widetilde \sigma$ satisfy the assumptions (B1) to (B3) and Theorem~\ref{Thm1} we obtain that there exist $c_1, c_2\in (0, \infty)$ such that for all $n\in\N$,
\begin{align*}
\EE\bigl[\|X-G_{\xi,\alpha,\nu}^{-1}\circ \eultr_{n}\|_\infty^p\bigr]^{1/p}\leq c_1\cdot \EE\bigl[\|Z- \eultr_{n}\|_\infty^p\bigr]^{1/p}\leq \frac{c_2}{ n^{3/4}},
\end{align*}
which completes the proof of the theorem.
\end{proof}

\section{Discussion of the error bounds in Theorems~\ref{Thm1} and~\ref{Thm2}}\label{disc}

It is well known that, in general, the upper error bound~\eqref{ll33} in Theorem~\ref{Thm1} can not be improved by any method that is based on $n$ evaluations of the driving Brownian motion $W$,
see~\cite{hhmg2019,m04} for results on matching lower error bounds.

We believe that an analogue statement holds true with respect to the upper error bound~\eqref{ll3} in Theorem~\ref{Thm1}. In particular, we conjecture that the following statement is true:

\begin{conj}\label{Conj1}
There exist $x_0\in\R$, functions $\mu,\sigma\colon\R\to\R$ that satisfy (B1) to (B3) 
and $c\in (0,\infty)$  such that the  solution $X$ of the corresponding SDE~\eqref{sde01} satisfies for every $n\in\N$,
\begin{equation}\label{conj1}
\inf_{\substack{t_1,\dots,t_n\in[0,1]  \\ \psi\colon\R^n\to \R\text{ measurable}}} \EE\bigl[|X_1-\psi(W_{t_1},\dots,W_{t_n})|\bigr] \ge \frac{c}{ n^{3/4}}. 
\end{equation} 
\end{conj}

We furthermore believe that in general the upper error
bound~\eqref{l3} 
in Theorem~\ref{Thm2} can not be improved by any method that is based on $n$  evaluations of the driving Brownian motion $W$. In particular, we conjecture that the following statement is true:

\begin{conj}\label{Conj2}
There exist $x_0\in\R$, functions $\mu,\sigma\colon\R\to\R$ that satisfy (A1) to (A3)  and $c\in (0,\infty)$  such that the solution $X$ of the corresponding  SDE~\eqref{sde0} satisfies for every $n\in\N$,
\begin{equation}\label{conj2}
\inf_{\substack{t_1,\dots,t_n\in[0,1]  \\ \psi\colon\R^n\to \R\text{ measurable}}} \EE\bigl[|X_1-\psi(W_{t_1},\dots,W_{t_n})|\bigr] \ge \frac{c}{ n^{3/4}}. 
\end{equation}
\end{conj}

Note that the assumptions (B1) to (B3) are stronger than the assumptions (A1) to (A3). Thus, if the Conjecture~\ref{Conj1} is true then the Conjecture~\ref{Conj2} is true as well.

On the other hand side, the following example shows that the lower bound~\eqref{conj2}
does not hold true for all choices of the coefficients $\mu,\sigma\colon\R\to\R$ that satisfy (A1) to (A3) such that $\mu$ is discontinuous.

\begin{ex}\label{ex2} Let $x_0=0$, take $k=1$, $z = 0$, $\alpha=-1/2$ and $\nu\in (0,1/4)$ in~\eqref{fct1},  and consider the functions $\mu,\sigma\colon\R\to\R$ given by
\begin{equation}\label{choice33}
\mu =1_{[0,\infty)},\quad \sigma = 1/G'_{0,-1/2,\nu}.
\end{equation}
Clearly, $\mu$ is Lipschitz continuous and differentiable on  each of the intervals $(-\infty,0)$ and $(0,\infty)$ with Lipschitz continuous derivative $\mu'=0$.  Using Lemma~\ref{lemx1}(i) we see that $\sigma$ is Lipschitz continuous on $\R$. Moreover, by Lemma~\ref{lemx1}(ii) we obtain that on each of the intervals $(-\infty,0)$ and $(0,\infty)$, $\sigma$ is differentiable  with  derivative $\sigma' = -G''_{0,-1/2,\nu}/(G'_{0,-1/2,\nu})^2$. According to Lemma~\ref{lemx1}(i) there exist $c\in (0,\infty)$ such that 
for all $x\in \R$ with $|x|>c$ we have $G''_{0,-1/2,\nu}(x)=0$. Employing Lemma~\ref{mult} in Section~\ref{Lem} we thus conclude that on each of the intervals $(-\infty,0)$ and $(0,\infty)$, $\sigma'$ is Lipschitz continuous. Moreover, by Lemma~\ref{lemx1}(i) we have $\sigma(0) = 1$. Hence, $\mu$ and $\sigma$ satisfy (A1) to (A3) with $k=1$ and $\xi_1=0$. 

Let $X$ denote the solution of~\eqref{sde0} with $\mu$ and $\sigma$ given by~\eqref{choice33}. Since $\sigma(0) = 1$ we obtain by Lemmas~\ref{transform1},~\ref{Ito} that the process $Z = G_{0,-1/2,\nu}\circ X$ is the solution of the SDE~\eqref{sde1} with coefficients $\widetilde \mu,\widetilde \sigma\colon \R\to\R$ that satisfy (B1) to (B3). Note that $ \widetilde \sigma = 1$ and thus one may take $k_{\widetilde\sigma}=0$
 in (B2). Hence, using the second part of Theorem~\ref{Thm1} and the Lipschitz continuity of $G^{-1}_{0,-1/2,\nu}$ we conclude that for every $p\in[1,\infty)$ there exist $c_1,c_2\in (0,\infty)$ such that for every $n\in\N$,
\[
\EE[\|X-G^{-1}_{0,-1/2,\nu}\circ \eultr_{n}\|_\infty^p]^{1/p} \le c_1\cdot \EE[\|Z-\eultr_{n}\|_\infty^p]^{1/p} \le c_2/n.
\]
\end{ex}

\section{Proof of  Theorem~\ref{Thm1}}\label{Proofs}
Throughout this section we 
assume that $\mu,\sigma\colon \R\to \R$ satisfy the 
assumptions
(B1) to (B3). 
Moreover, we put
\[
\utn = \lfloor n\cdot t\rfloor / n
\]
for every $n\in\N$ and every $t\in [0,1]$.

We briefly describe the 
structure
of this section. 
In Section~\ref{4.1} we provide $L_p$-estimates and a Markov property
of the time-continuous quasi-Milstein scheme $\eul_n$. Section~\ref{4.2} contains occupation time estimates for  $\eul_n$, which finally lead to the $p$-th mean estimate
\begin{equation}\label{main}
\max_{\xi\in\{\xi_1,\dots,\xi_{k_\mu}\}\cup\{\eta_1,\dots,\eta_{k_\sigma}\}}\EE\biggl[\Bigl|\int_0^1(\eul_{n,t}-\eul_{n,\utn})^2\cdot 1_{\{(\eul_{n,t} - \xi)(\eul_{n,\utn} - \xi)\le 0\}}\, dt\Bigr|^p\biggr] \le \frac{c}{n^{3p/2}},
\end{equation}
where $c\in (0,\infty)$ does not depend on $n$, see Proposition~\ref{prop1}. The latter result is a crucial tool for the error analysis of the quasi-Milstein scheme. The results in Sections~\ref{4.1} and~\ref{4.2} are then used in Section~\ref{4.3} to derive the error estimates in Theorem~\ref{Thm1}.

Throughout this section
we will employ the following three facts, which are an immediate consequence of the assumptions (B1) to (B3). Namely, 
the functions $\mu$ and $\sigma$ satisfy a linear growth condition, i.e. 
\begin{equation}\label{LG}
\exists\, K\in (0, \infty)\,\forall\, x\in\R\colon\quad |\mu(x)|+|\sigma(x)|\leq K\cdot (1+|x|),
\end{equation}
the functions  $\delta_\mu$ and $\delta_\sigma$ are bounded, i.e. 
\begin{equation}\label{bound}
\|\delta_\mu\|_\infty + \|\delta_\sigma\|_\infty < \infty,
\end{equation}
the function $\mu$
and $\sigma$ satisfy 
\begin{equation}\label{taylor}
\begin{aligned}
& \exists\, c\in(0, \infty)\,\forall\, 
i\in\{1,\dots,k_\mu+1\}\,
\forall x,y\in
 (\xi_{i-1},\xi_{i})\,
\forall\, 
j\in\{1,\dots,k_\sigma+1\}
\,\forall \tilde x,\tilde y\in 
(\eta_{j-1},\eta_{j})
\colon \\
& \qquad \qquad\qquad |\mu(y)-\mu(x)-\mu'(x)(y-x)|  \le c\cdot |y-x|^2\,\text{ and }\\
& \qquad \qquad\qquad |\sigma(\tilde y)-\sigma(\tilde x)-\sigma'(\tilde x)(\tilde y-\tilde x)|  \leq c\cdot |\tilde y-\tilde x|^2.
\end{aligned}
\end{equation}

\subsection{$L_p$-estimates and a Markov property for the time-continuous quasi-Milstein scheme}\label{4.1}

For technical reasons we have to provide $L_p$-estimates and  further properties
of the time-continuous quasi-Milstein scheme for the SDE~\eqref{sde01} dependent on the initial value $x_0$. To be formally precise, for every $x\in\R$ we let 
$X^x$ denote the unique strong solution of the SDE
\begin{equation}\label{sde00}
\begin{aligned}
dX^x_t & = \mu(X^x_t) \, dt + \sigma(X^x_t) \, dW_t, \quad t\in [0,1],\\
X^x_0 & = x,
\end{aligned}
\end{equation}
and for all $x\in\R$ and $n\in\N$ we use $\eul_{n}^x=(\eul_{n,t}^x)_{t\in[0,1]}$ to denote the time-continuous quasi-Milstein scheme with step-size $1/n$ associated to the SDE \eqref{sde00}.
Thus, $X=X^{x_0}$ and $\widehat X_n = \eul_{n}^{x_0}$ for all $n\in\N$, and for all $x\in\R$
\begin{equation}\label{intrep}
\eul_{n,t}^x=x+\int_0^t \mu(\eul_{n,\usn}^x)\, ds+\int_0^t\bigl(\sigma(\eul_{n,\usn}^x)+\sigma\delta_\sigma(\eul_{n,\usn}^x)\cdot (W_s-W_{\usn})\bigr)\,dW_s
\end{equation}
holds $\PP\text{-a.s.}$ for all  $n\in\N$ and $t\in[0,1]$.

We have the following uniform $L_p$-estimates for $\eul_{n}^x$, $n\in\N$,
which follow from \eqref{intrep}, the linear growth property~\eqref{LG} of $\mu$ and $\sigma$ 
and the boundedness of $\delta_\sigma$, see~\eqref{bound}, 
by using standard arguments.

\begin{lemma}\label{eulprop}
Let $p\in[1, \infty)$. Then there exists  $c\in(0, \infty)$ such that for all $x\in\R$, all $n\in\N$, all $\delta\in[0,1]$ and all $t\in[0, 1-\delta]$,
\[
\EE\Bigl[\,\sup_{s\in[t, t+\delta]} |\eul_{n,s}^x-\eul_{n,t}^x|^p\Bigr]^{1/p}\leq c\cdot (1+|x|)\cdot \sqrt{\delta}.
\]
In particular, there exists  $c\in(0, \infty)$ such that for all $x\in\R$ and all $n\in\N$,
\[
\sup_{n\in\N}\, \EE\bigl[\|\eul_{n}^x\|_\infty^p\bigr]^{1/p}\leq c\cdot (1+|x|).
\]
\end{lemma}

The following lemma provides a Markov property of the time-continuous quasi-Milsein
scheme
$\eul^x_n$ relative to the gridpoints $0,1/n,2/n,\ldots,1$.

\begin{lemma}\label{markov}
For all $x\in\R$, all $n\in\N$, all $j\in\{0, \ldots, n-1\}$ and $\PP ^{\eul_{n,j/n}^x} $-almost all $y\in\R$ we have
\[
\PP ^{(\eul_{n,t}^x)_{t\in [j/n, 1]}|\mathcal F_{j/n}}=\PP ^{(\eul_{n,t}^x)_{t\in [j/n, 1]}|\eul_{n,j/n}^x}
\]
as well as
\[
\PP ^{(\eul_{n,t}^x)_{t\in [j/n, 1]}|\eul_{n,j/n}^x=y}=\PP ^{(\eul^y_{n,t})_{t\in [0,1-j/n]}}.
\]
\end{lemma}
\begin{proof}
The lemma is an immediate consequence of the fact that, by definition of $\eul_n^x$, for every $\ell\in \{1,\ldots,n\}$  there exists a measurable mapping 
$\psi\colon \R\times C([0,\ell/n]) \to C([0,\ell/n]) $ such that for all $x\in\R$ and all $i\in\{0,1,\ldots,n-\ell\}$,
\[
(\eul^x_{n,t+i/n})_{t\in[0, \ell/n]} = \psi\bigl(\eul^x_{n, i/n},(W_{t+i/n}-W_{i/n})_{t\in[0, \ell/n]}\bigr).\qedhere
\]
\end{proof}

\subsection{Occupation time estimates for the time-continuous quasi-Milstein scheme}\label{4.2}

We first provide an estimate for the expected occupation time  of a neighborhood of   a non-zero of $\sigma$ by the time-continuous quasi-Milstein scheme $\eul_{n}^x$.

\begin{lemma}\label{occup}
Let $\xi\in\R$ satisfy $\sigma(\xi)\not=0$. Then there exists $ c\in (0, \infty)$ such that for   all $x\in\R$, all $n\in\N$ and all $\eps\in (0,\infty)$,
\begin{equation}
\int_0^1  \PP(\{|\eul_{n,t}^x-\xi|\leq \varepsilon\})\,dt\leq c\cdot (1+x^2)\cdot\Bigl(\varepsilon+\frac{1}{\sqrt n}\Bigr). 
\end{equation}
\end{lemma}

\begin{proof}
Let  $x\in\R$ and $n\in\N$. For $t\in[0,1]$ put
\[
\Sigma^x_{n,t}=\sigma(\eul_{n,\utn}^x)+\sigma \delta_\sigma (\eul_{n,\utn}^x)\cdot (W_t-W_{\utn}).
\]
Using~\eqref{LG},~\eqref{bound},
\eqref{intrep}
and Lemma~\ref{eulprop} we conclude that $\eul_n^x$ is a continuous semi-martingale
 with quadratic variation 
\begin{equation}\label{qv}
\langle \eul_{n}^x\rangle_t
=\int_0^t (\Sigma^x_{n,s})^2\, ds,\quad t\in[0,1].
\end{equation}
For $a\in\R$ let $L^a(\eul_n^x) = (L^a_t(\eul_n^x))_{t\in[0,1]}$ denote the local time
of $\eul_{n}^x$ at the point $a$.
Thus, 
for all $a\in\R$ and
 all $t\in[0,1]$, 
\begin{align*}
|\eul_{n,t}^x-a| & = |x-a| + \int_0^t \sgn(\eul_{n,s}^x-a)\cdot \mu (\eul_{n,s}^x)\, ds + \int_0^t \sgn(\eul_{n,s}^x-a)\cdot \Sigma^x_{n,s} \, dW_s + L^a_t(\eul_n^x),
\end{align*}
where $\sgn(y) = 1_{(0,\infty)}(y) - 1_{(-\infty,0]}(y)$ for $y\in\R$,
see, e.g.~\cite[Chap.~VI]{RevuzYor2005}.
Hence, 
for all $a\in\R$ and
 all $t\in[0,1]$,
 \begin{equation}\label{jjj0}
\begin{aligned}
L^a_t(\eul_n^x) & \le |\eul^x_{n,t}-x| + \int_0^t |\mu (\eul_{n,s}^x)|\, ds + \Bigl|\int_0^t \sgn(\eul_{n,s}^x-a)\cdot \Sigma^x_{n,s} \, dW_s\Bigr|\\
&\leq 2\int_0^t |\mu (\eul_{n,s}^x)|\, ds +\Bigl|\int_0^t  \Sigma^x_{n,s} \, dW_s\Bigr|+ \Bigl|\int_0^t \sgn(\eul_{n,s}^x-a)\cdot \Sigma^x_{n,s} \, dW_s\Bigr|.
\end{aligned}
\end{equation}

Using~\eqref{LG},~\eqref{jjj0}, the H\"older inequality and the Burkholder-Davis-Gundy inequality we obtain that
there exists $c\in (0,\infty)$ such that
  for all $x\in\R$, all $n\in\N$,
all $a\in\R$ 
   and all $t\in[0,1]$,
 
\begin{equation}\label{jjj1}
\EE\bigl[L^a_t(\eul_n^x)\bigr]  \le c\cdot \int_0^1 \bigl(1+\EE\bigl[|\eul_{n,s}^x|\bigr]\bigr)\, ds+c \,\Bigl(\int_0^1  \EE\bigl[(\Sigma^x_{n,s})^2\bigr] \, ds\Bigr)^{1/2}.
\end{equation}  
By \eqref{LG},~\eqref{bound}
and the fact that for all $s\in[0,1]$ the random variables $\eul_{n,\usn}^x$ and $W_s-W_{\usn}$ are independent we obtain that there exist $c_1, c_2\in(0,\infty)$ such that for all $s\in[0,1]$, all $x\in\R$ and all $n\in\N$,
\begin{equation}\label{jjj2}
\EE\bigl[(\Sigma^x_{n,s})^2\bigr]\leq c_1\cdot \EE\bigl[(1+|\eul_{n,\usn}^x|)^2\cdot (1+|W_s-W_{\usn}|)^2\bigr]\leq c_2\cdot\bigl(1+\EE\bigl[ (\eul_{n,\usn}^x)^2\bigr]\bigr).
\end{equation}
By employing Lemma~\ref{eulprop} we conclude from~\eqref{jjj1} and~\eqref{jjj2} 
that there exist 
$c_1, c_2 \in (0, \infty)$ such that
  for all $x\in\R$, all $n\in\N$,
all $a\in\R$ 
   and all $t\in[0,1]$,
\begin{equation}\label{local1}
\EE\bigl[L^a_t(\eul_n^x)\bigr] 
\leq  c_1\cdot \Bigl(1+
\EE\bigl[\|\eul_{n}^x\|_\infty^2\bigr]^{1/2}\Bigr)\leq  c_2\cdot (1+|x|).
\end{equation}
Using~\eqref{qv},~\eqref{local1} and the occupation 
time
formula it follows that
there exists $c\in (0,\infty)$ such that  
for all $x\in\R$, all $n\in\N$ and all $\eps\in (0,\infty)$,
\begin{equation}\label{local2}
  \EE\biggl[\int_0^1 1_{[\xi-\eps,\xi+\eps]}(\eul^x_{n,t})\cdot (\Sigma^x_{n,t})^2\, dt\biggr]
 = \int_{\R}1_{[\xi-\eps,\xi+\eps]}(a)\cdot \EE\bigl[L^a_t(\eul_n^x)\bigr]\, da 
 \le c\cdot (1+|x|)\cdot \eps. 
\end{equation}

By ~\eqref{LG},~\eqref{bound}
and the Lipschitz continuity of $\sigma$
we obtain that 
there exist  $c_1,c_2\in (0,\infty)$ such that 
for all $x\in\R$, all $n\in\N$ and all $t\in[0,1]$,
\begin{align*}
\bigl|\sigma^2(\eul^x_{n,t})-(\Sigma^x_{n,t})^2\bigr|&\leq \bigl|\sigma(\eul^x_{n,t})-\Sigma^x_{n,t}\bigr|\cdot \bigl(|\sigma(\eul^x_{n,t})|+|\Sigma^x_{n,t}|\bigr)\\
&\leq c_1\cdot \bigl(|\sigma(\eul^x_{n,t})-\sigma(\eul_{n,\utn}^x)|+|\sigma \delta_\sigma (\eul_{n,\utn}^x)|\cdot |W_t-W_{\utn}|\bigr)\\
&\qquad\,\cdot \bigl(1+|\eul^x_{n,t}|+(1+|\eul^x_{n,\utn}|)\cdot(1+|W_t-W_{\utn}|)\bigr)\\
&\leq c_2\cdot  \bigl(|\eul^x_{n,t}-\eul_{n,\utn}^x|+(1+|\eul_{n,\utn}^x|)\cdot |W_t-W_{\utn}|\bigr) \\
&\qquad\,\cdot(1+\|\eul^x_{n}\|_\infty)\cdot(1+|W_t-W_{\utn}|).
\end{align*}
Thus, using the H\"older inequality
and Lemma~\ref{eulprop}  we conclude that there exist  $c\in (0,\infty)$ such that 
for all $x\in\R$, all $n\in\N$ and all $t\in[0,1]$,
\begin{equation}\label{local3} 
\EE\bigl[|\sigma^2(\eul^x_{n,t})-(\Sigma^x_{n,t})^2|\bigr]\leq c\cdot (1+x^2)\cdot \frac{1}{\sqrt n}.
\end{equation}

Since $\sigma$ is continuous and $\sigma(\xi)\neq 0$ there exist  $\kappa,\eps_0\in(0,\infty)$ such that 
\begin{equation}\label{localx} 
\inf_{z\in\R: |z-\xi|< \eps_0}\sigma^2(z) \ge \kappa.
\end{equation}
Using~\eqref{local2},~\eqref{local3} and~\eqref{localx} we obtain that there exists $c\in (0,\infty)$ such that for all $x\in\R$, all $n\in\N$ and all $\eps \in (0,\eps_0]$,
\begin{align*}
 \int_0^1  \PP(\{|\eul_{n,t}^x-\xi|\leq \varepsilon\})\,dt  &  = \frac{1}{\kappa}\cdot \EE\Bigl[\int_0^1 \kappa\cdot 1_{[\xi-\eps,\xi+\eps]}(\eul^x_{n,t})\, dt\Bigr] \\ &  \le \frac{1}{\kappa}\cdot \EE\Bigl[\int_0^1 1_{[\xi-\eps,\xi+\eps]}(\eul^x_{n,t})\cdot  \sigma^2(\eul^x_{n,t})\, dt\Bigr] \\
& \le  \frac{1}{\kappa}\cdot \EE\biggl[\int_0^1\Bigl( 1_{[\xi-\eps,\xi+\eps]}(\eul^x_{n,t})\cdot  (\Sigma^x_{n,t})^2 + \bigl|\sigma^2(\eul^x_{n,t})-(\Sigma^x_{n,t})^2\bigr|\Bigr)\, dt\biggr]\\
&  \le  \frac{c}{\kappa}\cdot (1+|x|+x^2)\cdot \Bigl( \eps + \frac{1}{\sqrt n}\Bigr),
\end{align*}  
which completes the proof of the lemma.
\end{proof}

The following result provides moment estimates subject to the condition of a sign change of the process $\eul_n -\xi$ at time $t$ relative to its sign at the grid point $\utn$. 

\begin{lemma}\label{central}
   Let $q\in [1,\infty)$, $\xi\in\R$, and let
\[
A_{n,t} =\{(\eul_{n,t}-\xi)\cdot(\eul_{n, \utn}-\xi)\leq 0\}
\]
for all $n\in\N$ and $t\in[0,1]$. 
Then there exists $c\in(0, \infty)$ such that for all $n\in\N$, all $0\le s\le t \le 1$ with $\utn-s \ge 1/n$ and all real-valued, non-negative, $\F_s$-measurable random
variables  $Y$,
\begin{equation}\label{central1}
\begin{aligned}
& \EE\bigl[Y\cdot |W_t-W_{\utn}|^q\cdot \ind_{A_{n,t}}\bigr]\\
& \qquad\qquad \le \frac{c}{n^{q/2+1}}\cdot \EE[Y] + \frac{c}{n^{q/2}}\cdot \int_{\R} |z|^q\cdot \EE\bigl[Y\cdot \ind_{ \{|\widehat X_{n, \utn-(t-\utn)}-\xi| \le \tfrac{c}{\sqrt{n}} (1+|z|)\}}\bigr]\cdot e^{-\frac{z^2}{2}}\, dz.
\end{aligned}
\end{equation}
\end{lemma}

\begin{proof} 
Choose $K\in (0,\infty)$ according to~\eqref{LG}, put
\[
\kappa=K\cdot (1+|\xi|) \cdot(1+\|\delta_\sigma\|_\infty)
\]
 and choose $n_0\in\N\setminus\{1,2\}$ 
such that for all $n \geq n_0$,
\begin{equation}\label{n0}
\frac{16 \ln(n)}{\sqrt n}\leq 1 \qquad \text{and}\qquad 8\kappa\cdot\frac{1+2\sqrt{\ln(n)}}{\sqrt{n}}\le \frac{1}{2}.
\end{equation}
Without loss of generality we may assume that $n\ge n_0$.
Let $ 0\le s\le t \le 1$ with $\utn-s \ge 1/n$ and let $Y$ be a real-valued, 
non-negative, $\F_s$-measurable random variable.
If $t=\utn$ then \eqref{central1} trivially holds 
for any $c\in (0,\infty)$.

Now assume that $t > \utn$ and put 
\[
Z_1 = \frac{W_t-W_{\utn}}{\sqrt{t-\utn}},\quad Z_2 = \frac{W_{\utn}-W_{\utn-(t-\utn)}}{\sqrt{t-\utn}},\quad Z_3=\frac{W_{\utn-(t-\utn)}-W_{\utn-1/n} }{\sqrt{1/n-(t-\utn)}}.
\]
Below we show that 
\begin{equation}\label{central2}
\begin{aligned}
&A_{n,t}\cap \bigl\{\max_{i\in\{1,2,3\}}|Z_i| \le 2\sqrt{\ln(n)}\bigr\}   \subset \bigl\{|\widehat X_{n, \utn-(t-\utn)}-\xi| \le 8\kappa\cdot (1+|Z_1|+|Z_2|)/\sqrt{n}\bigr\}. 
\end{aligned}
\end{equation}
Note that $Z_1,Z_2,Z_3$ are independent and identically  distributed standard normal random variables. Moreover, $(Z_1,Z_2,Z_3)$ is independent of $\F_s$ since $s \le \utn-1/n$, $(Z_1,Z_2)$ is independent of $\F_{\utn-(t-\utn)}$ and $\widehat X_{n, \utn-(t-\utn)}$ is $\F_{\utn-(t-\utn)}$-measurable. Using the latter three facts jointly
 with~\eqref{central2} and a standard estimate of standard normal tail probabilities we  obtain that
\begin{align*}
& \EE\bigl[Y\cdot |W_t-W_{\utn}|^q\cdot \ind_{A_{n,t}}\bigr]\\
& \qquad= (t-\utn)^{q/2}\cdot\EE\bigl[Y\cdot |Z_1|^q\cdot \ind_{A_{n,t}}\bigr]\\
& \qquad\le \frac{1}{n^{q/2}}\cdot \EE\bigl[Y\cdot |Z_1|^q\cdot \ind_{\{|\widehat X_{n, \utn-(t-\utn)}-\xi| \le 8\kappa\cdot (1+|Z_1|+|Z_2|)/\sqrt{n}\}}\bigr] \\
& \qquad \qquad   + \frac{1}{n^{q/2}}\cdot \EE\bigl[Y\cdot |Z_1|^q\cdot \ind_{\{\max_{i\in\{1,2,3\}}|Z_i| > 2\sqrt{\ln(n)}\}}\bigr]\\
&  \qquad  = \frac{2}{\pi n^{q/2}}\int_{[0,\infty)^2} \EE\bigl[Y\cdot z_1^q\cdot \ind_{\{|\widehat X_{n, \utn-(t-\utn)}-\xi| \le 8\kappa\cdot (1+z_1+z_2)/\sqrt{n}\}}\bigr]\cdot e^{-\frac{z_1^2+z_2^2}{2}}\, d(z_1,z_2)\\
& \qquad \qquad   + \frac{1}{n^{q/2}}\cdot \EE[Y]\cdot \EE\bigl[|Z_1|^q\cdot \ind_{\{\max_{i\in\{1,2,3\}}|Z_i| > 2\sqrt{\ln(n)}\}}\bigr]\\
&  \qquad  \le \frac{2^{q/2+1}}{\pi n^{q/2}}\int_{\R^2}\EE\bigl[Y\cdot \bigl(\tfrac{|z_1+z_2|}{\sqrt 2}\bigr)^q\cdot \ind_{\{|\widehat X_{n, \utn-(t-\utn)}-\xi| \le 8\sqrt 2\kappa\cdot (1+|z_1+z_2|/\sqrt 2)/\sqrt{n}\}}\bigr]\cdot e^{-\frac{z_1^2+z_2^2}{2}}\, d(z_1,z_2)\\
& \qquad \qquad   + \frac{1}{n^{q/2}}\cdot \EE[Y]\cdot \EE\bigl[Z_1^{2q}\bigr]^{1/2}\cdot \bigl(\PP\bigl(\bigl{\{}\max_{i\in\{1,2,3\}}|Z_i| > 2\sqrt{\ln(n)}\bigr{\}}\bigr)\bigr)^{1/2}\\
&  \qquad  \leq \frac{2^{q/2+2}}{\sqrt{2\pi}n^{q/2}}\int_{\R} \EE\bigl[Y\cdot |z|^q\cdot \ind_{\{|\widehat X_{n, \utn-(t-\utn)}-\xi| \le 8\sqrt 2\kappa\cdot (1+|z|)/\sqrt{n}\}}\bigr]\cdot e^{-\frac{z^2}{2}}\, dz\\
 & \qquad \qquad + \frac{\sqrt{3\cdot 5\cdots (2q-1)}}{n^{q/2}}\cdot \EE[Y]\cdot \Bigl(\frac{3}{2\sqrt{2\pi \ln(n)}\cdot n^2}\Bigr)^{1/2},
\end{align*}
which yields~\eqref{central1}. 

 It remains to prove the inclusion~\eqref{central2}. To this end let $\omega\in\Omega$ and assume that
\begin{equation}\label{central3}
\omega\in A_{n,t}\text{\quad and \quad} \max_{i\in\{1,2,3\}}|Z_i(\omega)| \le 2\sqrt{\ln(n)}.
\end{equation}
By \eqref{LG} and~\eqref{central3}, 
\begin{equation}\label{abc}
\begin{aligned}
|\eul_{n,\utn}(\omega) -\xi| & \le |\eul_{n,t}(\omega)- \eul_{n,\utn}(\omega)| \\
&  = \Bigl|\mu(\eul_{n,\utn}(\omega))\cdot (t-\utn) + \sigma(\eul_{n,\utn}(\omega))\cdot \sqrt{t-\utn}\cdot Z_1(\omega)\\
&\qquad +\frac{1}{2}\sigma \delta_\sigma (\eul_{n,\utn}(\omega))\cdot (t-\utn)\cdot \bigl(Z_1^2(\omega)-1)\Bigr|\\
& \le K\cdot(1+|\eul_{n,\utn}(\omega)|)\cdot\Bigl(\frac{1}{n} +\frac{1}{\sqrt{n}}\cdot |Z_1(\omega)|+\frac{1}{\sqrt{n}}\cdot\|\delta_\sigma\|_\infty\cdot\frac{Z_1^2(\omega)+1}{2\sqrt n}\Bigr).
\end{aligned}
\end{equation}
Observe that for all $a,b\in\R$,
\begin{equation}\label{JJJ}
1+|a|\leq (1+|a-b|)\cdot (1+|b|).
\end{equation}
Moreover, \eqref{n0} and~\eqref{central3} yield
\begin{equation}\label{abc1}
\frac{Z_1^2(\omega)+1}{2\sqrt n}\leq \frac{4\ln(n)+1}{2\sqrt n}\leq \frac{5\ln(n)}{2\sqrt n} \leq 1.
\end{equation}
Combining~\eqref{abc} with~\eqref{abc1} and employing~\eqref{JJJ} with $a=\eul_{n,\utn}(\omega)$ and $b=\xi$ we get
\begin{equation}\label{LLL}
|\eul_{n,\utn}(\omega) -\xi|\leq  \frac{\kappa}{\sqrt{n}}\cdot (1+|\eul_{n,\utn}(\omega)-\xi|)\cdot (1+|Z_1(\omega)|).
\end{equation}
Similarly one can show that
\begin{equation}\label{central6}
|\eul_{n,\utn-(t-\utn)}(\omega)-\eul_{n,\utn-1/n}(\omega)| \le  \frac{\kappa}{\sqrt{n}}\cdot (1+|\eul_{n,\utn-1/n}(\omega)-\xi|)\cdot (1+|Z_3(\omega)|).
\end{equation} 
Furthermore, by \eqref{LG},
\begin{equation}\label{abc2} 
\begin{aligned}
& |\eul_{n,\utn}(\omega)-\eul_{n,\utn-(t-\utn)}(\omega)|\\
& \qquad\qquad  = \Bigl|\mu(\eul_{n,\utn-1/n}(\omega))\cdot (t-\utn) + \sigma(\eul_{n,\utn-1/n}(\omega))\cdot \sqrt{t-\utn}\cdot Z_2(\omega)\\
& \qquad\qquad \qquad+\frac{1}{2}\sigma\cdot \delta_\sigma (\eul_{n,\utn-1/n}(\omega))\cdot  \bigl(u-(t-\utn)\bigr)\Bigr|\\
&\qquad\qquad \leq K\cdot(1+|\eul_{n,\utn-1/n}(\omega)|)\cdot\Bigl(\frac{1}{n} +\frac{1}{\sqrt{n}}\cdot |Z_2(\omega)|+\frac{1}{2}\cdot\|\delta_\sigma\|_\infty\cdot\Bigl(|u|+\frac{1}{n}\Bigr)\Bigr),
\end{aligned}
\end{equation}
where
\[
u=(W_{\utn}(\omega)-W_{\utn-1/n}(\omega))^2 -(W_{\utn-(t-\utn)}(\omega)-W_{\utn-1/n}(\omega))^2.
\]
Observing that for all $a,b\in\R$,
\[
|(a+b)^2-b^2|\leq (|a|+|b|)^2,
\]
and using~\eqref{n0} as well as~\eqref{central3} we obtain
\begin{equation}\label{abc3}
|u|\leq  \bigl(\sqrt{t-\utn}\cdot |Z_2(\omega)|+\sqrt{1/n-(t-\utn)}\cdot |Z_3(\omega)|\bigr)^2\leq \frac{16 \ln(n)}{n}\leq \frac{1}{\sqrt n}.
\end{equation}
Combining~\eqref{abc2} with~\eqref{abc3} and employing~\eqref{JJJ} with $a=\eul_{n,\utn-1/n}(\omega) $ and $b=\xi$ we
conclude that
\begin{equation}\label{central5}
|\eul_{n,\utn}(\omega)-\eul_{n,\utn-(t-\utn)}(\omega)| \le \frac{\kappa}{\sqrt{n}}\cdot (1+|\eul_{n,\utn-1/n}(\omega)-\xi|)\cdot   (1+|Z_2(\omega)|).
\end{equation} 

Clearly,  we have
\begin{equation}\label{yyy}
|\eul_{n,\utn-(t-\utn)}(\omega) -\xi|\le |\eul_{n,\utn}(\omega)-\eul_{n,\utn-(t-\utn)}(\omega)| +|\eul_{n,\utn}(\omega)-\xi|.
\end{equation}
By~\eqref{n0} and~\eqref{central3}
we have  for all $i\in\{1,2,3\}$,
\begin{equation}\label{mmm}
 \frac{\kappa}{\sqrt{n}}\cdot (1+|Z_i(\omega)|)\le  \kappa\cdot \frac{1+ 2\sqrt{\ln(n)}}{\sqrt{n}} \leq \frac{1}{2}.
\end{equation}
Using \eqref{LLL} and \eqref{mmm} we obtain
\begin{equation}\label{central4}
|\eul_{n,\utn}(\omega)-\xi|\leq \frac{\frac{\kappa}{\sqrt{n}}\cdot (1+|Z_1(\omega)|)}{1-\frac{\kappa}{\sqrt{n}}\cdot (1+|Z_1(\omega)|)}\leq \frac{2\kappa}{\sqrt{n}}\cdot (1+|Z_1(\omega)|).
\end{equation}
Furthermore, by ~\eqref{central6} and \eqref{mmm},
\begin{align*}
 1 +|\eul_{n,\utn-(t-\utn)}(\omega) -\xi|  &\ge 1 + |\eul_{n,\utn-1/n}(\omega)-\xi| -|\eul_{n,\utn-(t-\utn)}(\omega)-\eul_{n,\utn-1/n}(\omega)|\\
&\ge (1 + |\eul_{n,\utn-1/n}(\omega)-\xi|)/2,
\end{align*} 
which jointly with \eqref{central5} yields 
\begin{equation}\label{bbb}
|\eul_{n,\utn}(\omega)-\eul_{n,\utn-(t-\utn)}(\omega)|\leq 
\frac{2\kappa}{\sqrt{n}}\cdot (1 +|\eul_{n,\utn-(t-\utn)}(\omega) -\xi|)\cdot   (1+|Z_2(\omega)|).
\end{equation}
Combining  \eqref{yyy}, \eqref{central4}  and \eqref{bbb} we conclude that
\begin{equation}\label{central8}
\begin{aligned}
& |\eul_{n,\utn-(t-\utn)}(\omega) -\xi| \le \frac{4\kappa}{\sqrt{n}}\cdot (1+|\eul_{n,\utn-(t-\utn)}(\omega)-\xi|)\cdot (1+|Z_1(\omega)|+|Z_2(\omega)|).
\end{aligned}
\end{equation} 
By~\eqref{n0} and~\eqref{central3},
\begin{equation}\label{abc5}
\frac{4\kappa}{\sqrt{n}}\cdot (1+|Z_1(\omega)|+|Z_2(\omega)|)\leq 8\kappa\cdot \frac{1+2\sqrt{\ln(n)}}{\sqrt n}\le \frac{1}{2}.
\end{equation}
Observing~\eqref{abc5} we obtain from~\eqref{central8} that
\[
|\eul_{n,\utn-(t-\utn)}(\omega) -\xi| \le  \frac{8\kappa}{\sqrt{n}}\cdot (1+|Z_1(\omega)|+|Z_2(\omega)|).
\]
This finishes the proof of~\eqref{central2}. 
\end{proof}

Using Lemmas~\ref{markov},~\ref{occup} and~\ref{central} we can now establish the following two estimates on time averages of moments subject to the condition of sign changes of $\eul_{n}-\xi$ relative to its sign at the gridpoints $0,1/n,\dots,1$.

\begin{lemma}\label{key}
Let $q\in [1,\infty)$, let $\xi\in\R$ satisfy $\sigma(\xi)\neq 0$ and let
\[
A_{n,t} =\{(\eul_{n,t}-\xi)\cdot(\eul_{n, \utn}-\xi)\leq 0\}
\]
for all $n\in\N$ and $t\in[0,1]$.  
Then there exists $c\in(0, \infty)$ such that for all $n\in\N$, all $s\in [0,1-1/n)$ and  all real-valued, non-negative, $\F_s$-measurable random variables  $Y$,
\begin{equation}\label{eee1}
 \int_{\usn+2/n}^1\EE \bigl[Y\cdot |W_t-W_{\utn}|^q\cdot \ind_{ A_{n,t}}\bigr]\, dt \le \frac{c}{n^{(q+1)/2}} \cdot \bigl( \EE[Y] + \EE\bigl[Y \cdot (\eul_{n,\usn+1/n}-\xi)^2\bigr]\bigr)
\end{equation}
and 
\begin{equation}\label{eee2}
\begin{aligned}
&\int_{\usn+1/n}^1\EE\bigl[Y\cdot |W_t-W_{\utn}|^ q\cdot \ind_{ A_{n,t}} \cdot (\eul_{n,\utn+1/n}-\xi)^2\bigr]\, dt\\
&\hspace{4cm} \le  \frac{c}{n^{q/2+1}} \cdot \bigl(\EE[Y] + \EE\bigl[Y \cdot (\eul_{n,\usn+1/n}-\xi)^2\bigr]\bigr).
\end{aligned}
\end{equation}
\end{lemma}

\begin{proof}
For $s\in[0,1]$ we use $\mathcal Y_s$ to denote the set of all  real-valued, non-negative, $\F_s$-measurable random variables.

We first prove~\eqref{eee1}. Note that if $t\geq\usn+2/n$ then  $\utn-1/n \ge \usn+1/n \ge s$. By Lemma~\ref{central} we thus obtain that there exists $c\in (0,\infty)$ such that for all $n\in\N$, $s\in [0,1-1/n)$ and $Y\in\mathcal Y_s$, 
\begin{equation}\label{key01}
\begin{aligned}
& \int_{\usn+2/n}^1\EE \bigl[Y\cdot |W_t-W_{\utn}|^q\cdot \ind_{ A_{n,t}}\bigr]\, dt\\
& \qquad\le c\cdot \frac{\EE[Y]}{n^{q/2+1}} +\frac{ c}{n^{q/2}} \int_{\R} |z|^q \cdot  e^{-\frac{z^2}{2}} \cdot \int_{\usn+2/n}^1 \EE\bigl[Y\cdot \ind_{ \{|\widehat X_{n, \utn-(t-\utn)}-\xi| \le \tfrac{c}{\sqrt{n}} (1+|z|)\}}\bigr]\,
\, dt\,dz \\
& \qquad =c\cdot\frac{\EE[Y]}{n^{q/2+1}} + \frac{c}{n^{q/2}} \int_{\R}  |z|^q\cdot  e^{-\frac{z^2}{2}}\cdot  \int_{\usn+1/n}^{1-1/n} \EE\bigl[Y\cdot \ind_{ \{|\widehat X_{n, t}-\xi| \le \tfrac{c}{\sqrt{n}} (1+|z|)\}}\bigr]\,
\, dt\,dz.
\end{aligned}
\end{equation}
 Using the fact that
for all $n\in\N$ and  $s\in [0,1-1/n)$
 every $Y\in\mathcal Y_s$ is $\F_{\usn+1/n}$-measurable and employing the first part of Lemma~\ref{markov} we obtain that for all  $n\in\N$, $s\in [0,1-1/n)$, $Y\in\mathcal Y_s$ and $z\in\R$,
\begin{equation}\label{key02}
\begin{aligned}
& \int_{\usn+1/n}^{1-1/n} \EE\bigl[Y\cdot \ind_{ \{|\widehat X_{n, t}-\xi| \le \tfrac{c}{\sqrt{n}} (1+|z|)\}}\bigr]\,
\, dt \\
& \qquad\qquad \qquad = \EE\Bigl[Y\cdot \EE\Bigl[\int_{\usn+1/n}^{1-1/n} \ind_{\{|\widehat X_{n,t}-\xi| \le \tfrac{c}{\sqrt{n}} (1+|z|)\}}\, dt\Bigl|\eul_{n,\usn+1/n}\Bigr]\Bigr].
\end{aligned}
\end{equation}
Moreover, by the second part of Lemma~\ref{markov} and by Lemma~\ref{occup},  there exists $c_1\in(0, \infty)$ such that  for all $n\in\N$, $s\in [0,1-1/n)$, $Y\in\mathcal Y_s$, $z\in\R$ and $P^{\eul_{n,\usn+1/n}}$-almost all $x\in\R$, 
\begin{equation}\label{key03}
\begin{aligned}
& \EE\Bigl[\int_{\usn+1/n}^{1-1/n} \ind_{\{|\widehat X_{n,t}-\xi| \le \tfrac{c}{\sqrt{n}} (1+|z|)\}}\, dt\Bigl|\eul_{n,\usn+1/n}=x\Bigr]\\
 & \qquad= \EE\Bigl[\int_{0}^{1-2/n-\usn} \ind_{\{|\widehat X^x_{n,t}-\xi| \le \tfrac{c}{\sqrt{n}} (1+|z|)\}}\, dt\Bigr] \le c_1\cdot (1+x^2)\cdot \Bigl( \frac{c}{\sqrt{n}} \cdot(1+|z|) + \frac{1}{\sqrt{n}}\Bigr).
\end{aligned}
\end{equation}
Combining~\eqref{key02} and~\eqref{key03} 
and using the fact that for all $a,b\in\R$,
\[
1+a^2\leq 2\,(1+(a-b)^2)\cdot (1+b^2),
\]
we conclude that  for  all $n\in\N$, $s\in [0,1-1/n)$, $Y\in\mathcal Y_s$ and $z\in\R$,
\begin{equation}\label{key04}
\begin{aligned}
& \int_{\usn+1/n}^{1-1/n} \EE\bigl[Y\cdot \ind_{ \{|\widehat X_{n, t}-\xi| \le \tfrac{c}{\sqrt{n}} (1+|z|)\}}\bigr]\, dt \\ 
& \qquad\qquad\qquad\qquad  \le \tfrac{c_1(c+1)}{\sqrt{n}}\cdot (1+|z|)\cdot \EE\bigl[Y\cdot (1 + \eul_{n,\usn+1/n}^2)\bigr]\\
 & \qquad\qquad\qquad\qquad  \le \tfrac{2c_1(c+1)}{\sqrt{n}}\cdot (1+\xi^2)\cdot (1+|z|)\cdot \bigl(\EE[Y] + \EE\bigl[Y\cdot (\eul_{n,\usn+1/n}-\xi)^2\bigr]\bigr).
\end{aligned}
\end{equation}
Inserting~\eqref{key04} into~\eqref{key01} and observing that $\int_\R (1+|z|) \cdot |z|^q \cdot e^{-z^2/2}\,dz < \infty$ completes the proof of~\eqref{eee1}. 

We next prove~\eqref{eee2}.
 Clearly, for all 
$n\in\N$, $s\in [0,1-1/n)$,
$t\in[\usn+1/n, 1]$ and all  $\omega\in A_{n,t}$ we have
\begin{align*}
|\eul_{n,\utn+1/n}(\omega)-\xi| &\le |\eul_{n,\utn+1/n}(\omega)-\eul_{n,t}(\omega)| + |\eul_{n,t}(\omega)-\xi| \\
&\le |\eul_{n,\utn+1/n}(\omega)-\eul_{n,t}(\omega)| + |\eul_{n,t}(\omega)-\eul_{n,\utn}(\omega)|.
\end{align*}
 Using the fact that 
for all $n\in\N$ and  $s\in [0,1-1/n)$
every $Y\in\mathcal Y_s$ is $\F_{\usn+1/n}$-measurable and employing the   H\"older inequality  we therefore obtain that for all $n\in\N$, all $s\in [0,1-1/n)$, all $Y\in\mathcal Y_s$ and all $t\in[\usn+1/n,1]$,
\begin{equation}\label{key05}
\begin{aligned}
&\EE\bigl[Y\cdot |W_t-W_{\utn}|^q\cdot \ind_{ A_{n,t}} \cdot (\eul_{n,\utn+1/n}-\xi)^2\bigr] \\
 &  \qquad\quad  \le \EE \bigl[Y\cdot |W_t-W_{\utn}|^q\cdot (|\eul_{n,\utn+1/n}-\eul_{n,t}| + |\eul_{n,t}-\eul_{n,\utn}|)^2 \bigr]\\ 
 &  \qquad\quad  = \EE \bigl[Y\cdot \EE\bigl[|W_t-W_{\utn}|^q\cdot (|\eul_{n,\utn+1/n}-\eul_{n,t}| + |\eul_{n,t}-\eul_{n,\utn}|)^2|\F_{\usn+1/n}\bigr] \bigr]\\ 
& \qquad\quad  \leq \EE\bigl[ Y\cdot\bigl(\EE\bigl[(W_t-W_{\utn})^{2q}|\F_{\usn+1/n}\bigr]\bigr)^{1/2}\\
&\hspace{3cm} \cdot \bigl(\EE\bigl[(|\eul_{n,\utn+1/n}-\eul_{n,t}| + |\eul_{n,t}-\eul_{n,\utn}|)^4|\F_{\usn+1/n}\bigr]\bigr)^{1/2}\bigr].
\end{aligned}
\end{equation}
If $t\ge \usn+1/n$ then $\utn \ge \usn+1/n$. Hence, there exists $c\in (0,\infty)$ such that for all  $n\in\N$, all $s\in [0,1-1/n)$ and all $t\in[\usn+1/n,1]$,
\begin{equation}\label{sd}
\EE\bigl[(W_t-W_{\utn})^{2q}|\F_{\usn+1/n}\bigr]=\EE\bigl[(W_t-W_{\utn})^{2q}\bigr]\leq c/n^q.
\end{equation}
Moreover, the first part of Lemma~\ref{markov} implies that for
all $n\in\N$, all $s\in [0,1-1/n)$ and 
all $t\in[\usn+1/n,1]$ it holds
 $\PP\text{-a.s.}$ that
\begin{equation}\label{ds}
\begin{aligned}
&\EE\bigl[(|\eul_{n,\utn+1/n}-\eul_{n,t}| + |\eul_{n,t}-\eul_{n,\utn}|)^4|\F_{\usn+1/n}\bigr]\\
&\qquad\qquad=\EE\bigl[(|\eul_{n,\utn+1/n}-\eul_{n,t}| + |\eul_{n,t}-\eul_{n,\utn}|)^4| \eul_{n,\usn+1/n}\bigr].
\end{aligned}
\end{equation}
By the second part of Lemma~\ref{markov} and by Lemma~\ref{eulprop} we obtain that there exist $c_1, c_2\in(0, \infty)$ such that  for all $n\in\N$, all $s\in [0,1-1/n)$, all $t\in [\usn+1/n, 1]$ and
$\PP ^{\eul_{n,\usn+1/n}} $-almost all
 $x\in\R$,
\begin{equation}\label{key06}
\begin{aligned}
&  \EE\bigl[(|\eul_{n,\utn+1/n}-\eul_{n,t}| + |\eul_{n,t}-\eul_{n,\utn}|)^4\bigl| \eul_{n,\usn+1/n} = x\bigr]\\
& \qquad \quad = \EE[(|\eul^x_{n,\utn-\usn}-\eul^x_{n,t-\usn-1/n}| + |\eul^x_{n,t-\usn-1/n}-\eul^x_{\utn-\usn-1/n}|)^4\bigr] 
\\& \qquad\quad \le  c_1\cdot (1+x^4)\cdot 1/n^2
\le  c_2\cdot (1+(x-\xi)^4)\cdot 1/n^2.
\end{aligned}
\end{equation}
It follows from~\eqref{key05}, \eqref{sd}, \eqref{ds} and~\eqref{key06} that there exist $c\in(0, \infty)$ such that  for all $n\in\N$, all $s\in [0,1-1/n)$ and all $Y\in\mathcal Y_s$,
\begin{equation}\label{key07}
\begin{aligned}
& \int_{\usn+1/n}^1\EE\bigl[Y\cdot |W_t-W_{\utn}|^q\cdot \ind_{ A_{n,t}} \cdot (\eul_{n,\utn+1/n}-\xi)^2\bigr]\, dt \\
& \qquad\qquad  \le \frac{c}{n^{q/2+1}}
\cdot \int_{\usn+1/n}^1 \EE\bigl[Y\cdot (1+(\eul_{n,\usn+1/n}-\xi)^2)\bigr]\, dt\\
& \qquad\qquad \le \frac{c}{n^{q/2+1}}
\cdot \bigl(\EE[Y] + \EE\bigl[Y\cdot (\eul_{n,\usn+1/n}-\xi)^2\bigr]\bigr),
\end{aligned}
\end{equation}
which finishes the proof of~\eqref{eee2} and completes the proof of the lemma. 
\end{proof}

We are ready to establish the main result in this section, which provides a $p$-th mean estimate of the time 
average of  $|\eul_{n,t}-\eul_{n,\utn}|^q$
subject to a sign change of $\eul_{n,t}-\xi$ relative to the sign of $\eul_{n,\utn}-\xi$.

\begin{prop}\label{prop1}
Let $\xi\in\R$ satisfy $\sigma(\xi)\not=0$ and let
\[
A_{n,t} =\{(\eul_{n,t}-\xi)\cdot(\eul_{n, \utn}-\xi)\leq 0\}
\]
for all $n\in\N$ and $t\in[0,1]$.  
Then for all  $p,q\in [1,\infty)$  there exists  $c\in(0, \infty)$ such that for all $n\in\N$, 
\begin{equation}\label{l33a}
\EE\Bigl[\Bigl|\int_0^1  |\eul_{n,t}-\eul_{n,\utn}|^q\cdot \ind_{A_{n,t}}\,dt\Bigr|^p\Bigr]^{1/p}\leq \frac{c}{n^{(q+1)/2}}. 
\end{equation}
\end{prop}

\begin{proof}
Clearly, it suffices to consider only the case $p\in\N$.
Fix $q\in [1,\infty)$. For $n,p\in\N$ put
\[
a_{n,p} = \EE\Bigl[\Bigl|\int_0^1  |W_t-W_{\utn}|^q\cdot \ind_{A_{n,t}}\, dt\Bigr|^p\Bigr].
\]
 We prove by induction on $p$ that for every $p\in\N$ there exists $c\in (0,\infty)$ such that for all $n\in\N$,
\begin{equation}\label{prop01}
a_{n,p} \le \frac{c}{n^{(q+1)p/2}}.
\end{equation}

First, consider the case $p=1$. Using~\eqref{eee1} in Lemma~\ref{key} with $s=0$ and $Y=1$ we 
obtain that there exists $c\in (0,\infty)$ such that for all $n\geq 2$,
\begin{align*}
a_{n,1} &\leq \int_{2/n}^1 \EE\bigl[ |W_t-W_{\utn}|^q\cdot \ind_{A_{n,t}}\bigr]\, dt +\int_0^{2/n} \EE\bigl[ |W_t-W_{\utn}|^q\bigr]\, dt\\
& \le \frac{c}{n^{(q+1)/2}}\cdot\bigl (1 + \EE\bigl[\bigl(\eul_{n,1/n}-\xi\bigr)^2 \bigr]\bigr)+\frac{c}{n^{q/2+1}}\\
& \le \frac{2c}{n^{(q+1)/2}}\cdot \bigl(1+\xi^2+\EE\bigl[\eul_{n,1/n}^2\bigr]\bigr).
\end{align*}
Employing Lemma~\ref{eulprop} we thus conclude that~\eqref{prop01} holds for $p=1$.

Next, let $r\in\N$ 
and assume that~\eqref{prop01} holds for all $p\in\{1, \ldots, r\}$. For $n\in\N$ and $t\in[0,1]$ put
\[
Y_{n,t}=|W_{t}-W_{\utn}|^q.
\]
We then have for all $n\in\N$,
\begin{equation}\label{i1}
\begin{aligned}
a_{n,r+1}  &= (r+1)!\cdot \int_0^1\int_{t_1}^1\ldots \int_{t_r}^1 \EE\Bigl[\prod_{i=1}^{r+1}Y_{n,t_i}\cdot \ind_{A_{n,t_i}}\Bigr]\,dt_{r+1}\, \ldots \, dt_1.
\end{aligned}
\end{equation}
For $n\in\N$ and $0\le t_1\le \ldots \le  t_r\le 1$  we put
\begin{align*}
J_{1,n}(t_1,\dots,t_r) & = \int_{t_r}^{(\utrn+2/n) \wedge 1} \EE\Bigl[\prod_{i=1}^{r+1}Y_{n,t_i}\cdot \ind_{A_{n,t_i}}\Bigr]\,dt_{r+1},\\
J_{2,n}(t_1,\dots,t_r) & = \int_{(\utrn+2/n) \wedge 1}^1 \EE\Bigl[\prod_{i=1}^{r+1}Y_{n,t_i}\cdot \ind_{A_{n,t_i}}\Bigr]\,dt_{r+1}.
\end{align*}
By the H\"older inequality
there exists $c\in (0,\infty)$ such that  for all $n\in\N$ and all $0\le t_1 \le \dots \le  t_{r+1}\le 1$,
\begin{equation}\label{i3}
\begin{aligned}
\EE\Bigl[\prod_{i=1}^{r+1}Y_{n,t_i}\cdot \ind_{A_{n,t_i}}\Bigr] & \le
\EE\Bigl[\Bigl(Y_{n,t_{r+1}}\cdot \prod_{i=1}^{r}Y_{n,t_i}^{1/r}\Bigr)\cdot\Bigl(\prod_{i=1}^{r}Y_{n,t_i}^{(r-1)/r}\cdot \ind_{A_{n,t_i}}\Bigr) \Bigr] \\
& \le \EE\Bigl[Y_{n,t_{r+1}}^r\cdot \prod_{i=1}^{r}Y_{n,t_i}\Bigr]^{1/r} \cdot\EE\Bigl[\prod_{i=1}^{r}Y_{n,t_i}\cdot \ind_{A_{n,t_i}}\Bigr]^{(r-1)/r}  \\
& \le \frac{c}{n^q}\cdot \EE\Bigl[\prod_{i=1}^{r}Y_{n,t_i}\cdot \ind_{A_{n,t_i}}\Bigr]^{(r-1)/r}.
\end{aligned}
\end{equation}
Hence there exists $c\in (0,\infty)$ such that for all $n\in\N$ and all $0\le t_1 \le \dots \le t_r\le 1$,
\begin{equation}\label{i4}
\begin{aligned}
J_{1,n}(t_1,\dots,t_r) & \le \frac{c}{n^{q+1}}\cdot \EE\Bigl[\prod_{i=1}^{r}Y_{n,t_i}\cdot \ind_{A_{n,t_i}}\Bigr]^{(r-1)/r}.
 \end{aligned}
\end{equation}
Clearly,   for all $n\in\N$ and all $0\le t_1\le \ldots \le  t_r\le 1$ with $t_r \ge 1-1/n$ we have
\begin{equation}\label{TTTTT}
J_{2,n}(t_1,\dots,t_r)=0.
\end{equation}
Furthermore, if $t_r\in [0,1-1/n)$ then $(\utrn+2/n) \wedge 1= \utrn+2/n$, and by applying~\eqref{eee1} in Lemma~\ref{key} with $s=t_r$ and $Y= \prod_{i=1}^{r}Y_{n,t_i}\cdot \ind_{A_{n,t_i}}$ we obtain that there exists $c\in (0,\infty)$ such that for all $n\in\N$ and all $0\le t_1\le \ldots \le  t_r\le 1$,
\begin{equation}\label{i2}
J_{2,n}(t_1,\dots,t_r) \le \frac{c}{n^{(q+1)/2}} \cdot \Bigl(\EE\Bigl[\prod_{i=1}^r Y_{n,t_i}\cdot \ind_{A_{n,t_i}}\Bigr] +\EE\Bigl[\prod_{i=1}^r Y_{n,t_i}\cdot \ind_{A_{n,t_i}}\cdot (\eul_{n,\utrn+1/n}-\xi)^2\Bigr]\Bigr).
\end{equation}
Combining \eqref{i4} to \eqref{i2} with~\eqref{i1} 
and employing the induction hypothesis we conclude that there exists $c_1, c_2 \in (0,\infty)$ such that for all $n\in\N$,
\begin{equation}\label{i5}
\begin{aligned}
a_{n,r+1}  
&  \le c_1\cdot\Bigl( \frac{a_{n,r}}{n^{(q+1)/2}} + \frac{b_{n,r}}{n^{(q+1)/2}} + \frac{a_{n,r}^{(r-1)/r}}{n^{q+1}}\Bigr) \le  c_2\cdot\Bigl(\frac{1}{n^{(q+1)(r+1)/2}} + \frac{b_{n,r}}{n^{(q+1)/2}}\Bigr),
 \end{aligned}
\end{equation}
where
\[
b_{n,r}  =  \int_0^1\int_{t_1}^1\ldots \int_{t_{r-1}}^1  \EE\Bigl[\Bigl(\prod_{i=1}^r Y_{n,t_i}\cdot \ind_{A_{n,t_i}}\Bigr)\cdot (\eul_{n,\utrn+1/n}-\xi)^2\Bigr]\,  dt_r\,\ldots \, dt_1.
\]

We proceed with estimating the term $b_{n,r}$.
Using~\eqref{eee2} in Lemma~\ref{key} with $s=0$ and $Y=1$ as well as Lemma~\ref{eulprop} we obtain that  there exist $c_1,c_2\in (0,\infty)$ such that  for all $n\in\N$,
\[
\int_{1/n}^1 \EE\bigl[ Y_{n,t_1}\cdot\ind_{A_{n,t_1}}\cdot (\eul_{n,\underline{t_1}_n+1/n} - \xi)^2\bigr]\, dt_1 \le \frac{c_1}{n^{q/2+1}}\cdot \bigl(1+ \EE[(\eul_{n,\underline{t_1}_n+1/n} - \xi)^2]\bigr) \le \frac{c_2}{n^{q/2+1}}.
\]
Furthermore, by employing Lemma~\ref{eulprop} again we see that there exist $c\in (0,\infty)$ such that for all $n\in\N$,
\begin{align*}
& \int_0^{1/n}\EE\bigl[ Y_{n,t_1}\cdot\ind_{A_{n,t_1}}\cdot (\eul_{n,\underline{t_1}_n+1/n} - \xi)^2\bigr]\, dt_1 \\
& \qquad\qquad\qquad \le \int_0^{1/n}\EE\bigl[ Y_{n,t_1}^2]^{1/2}\cdot \EE [(\eul_{n,\underline{t_1}_n+1/n} - \xi)^4\bigr]^{1/2}\, dt_1\le \frac{c}{n^{q/2+1}}.
\end{align*}
It follows that there exist $c\in (0,\infty)$ such that for all $n\in\N$,
\begin{equation}\label{i5c}
b_{n,1} \le \frac{c}{n^{q/2+1}}.
\end{equation}

Next, we assume that $r\ge 2$, and for $n\in\N$ and $0\le t_1\le \ldots \le  t_{r-1}\le 1$ we put
\begin{align*}
K_{1,n}(t_1,\dots,t_{r-1}) & = \int_{t_{r-1}}^{(\underline{t_{r-1}}_n+1/n) \wedge 1} \EE\Bigl[\Bigl(\prod_{i=1}^{r}Y_{n,t_i}\cdot \ind_{A_{n,t_i}}\Bigr)\cdot (\eul_{n,\utrn+1/n}-\xi)^2\Bigr]\,dt_{r},\\
K_{2,n}(t_1,\dots,t_{r-1}) & = \int_{(\underline{t_{r-1}}_n+1/n) \wedge 1}^1 \EE\Bigl[\Bigl(\prod_{i=1}^{r}Y_{n,t_i}\cdot \ind_{A_{n,t_i}}\Bigr)\cdot (\eul_{n,\utrn+1/n}-\xi)^2\Bigr]\,dt_{r}.
\end{align*}
Proceeding similarly
 to 
~\eqref{i3} and employing Lemma~\ref{eulprop} we conclude that there exists $c\in (0,\infty)$ such that for all $n\in\N$ and all $0\le t_1\le \ldots  \le t_r\le 1$,
\begin{equation}\label{i7}
\begin{aligned} 
& \EE\Bigl[\Bigl(\prod_{i=1}^r Y_{n,t_i}\cdot \ind_{A_{n,t_i}}\Bigr)
\cdot (\eul_{n,\underline{t_{r}}_n+1/n}-\xi)^2\Bigr]\\
& \qquad \le  \EE\Bigl[Y_{n,t_r}^{r-1}\cdot \Bigl(\prod_{i=1}^{r-1} Y_{n,t_i}\Bigr)
\cdot (\eul_{n,\underline{t_{r}}_n+1/n}-\xi)^{2(r-1)}
\Bigr]^{\tfrac{1}{r-1}}\cdot\EE\Bigl[\prod_{i=1}^{r-1} Y_{n,t_i}\cdot \ind_{A_{n,t_i}}\Bigr]^{\frac{r-2}{r-1}}\\
& \qquad \le \frac{c}{n^{q}}\cdot \EE\Bigl[\prod_{i=1}^{r-1} Y_{n,t_i}\cdot \ind_{A_{n,t_i}}\Bigr]^{\frac{r-2}{r-1}}.
 \end{aligned}
\end{equation}
Hence there exists $c\in (0,\infty)$ such that for all $n\in\N$ and all $0\le t_1 \le \dots \le t_{r-1}\le 1$,
\begin{equation}\label{i8}
\begin{aligned}
K_{1,n}(t_1,\dots,t_{r-1}) & \le \frac{c}{n^{q+1}}\cdot \EE\Bigl[\prod_{i=1}^{r-1}Y_{n,t_i}\cdot \ind_{A_{n,t_i}}\Bigr]^{\frac{r-2}{r-1}}.
 \end{aligned}
\end{equation}
Clearly,   for all $n\in\N$ and all $0\le t_1\le \ldots \le  t_{r-1}\le 1$ with $t_{r-1} \ge 1-1/n$ we have
\begin{equation}\label{JJJL}
 K_{2,n}(t_1,\dots,t_{r-1})=0.
\end{equation}
Furthermore,
if $t_{r-1}\in [0,1-1/n)$ then $(\underline{t_{r-1}}_n+1/n) \wedge 1= \underline{t_{r-1}}_n+1/n$, and by applying~\eqref{eee2} in Lemma~\ref{key} with $s=t_{r-1}$ and $Y= \prod_{i=1}^{r-1}Y_{n,t_i}\cdot \ind_{A_{n,t_i}}$ we obtain that there exists $c\in (0,\infty)$ such that for all $n\in\N$ and all $0\le t_1 \le \dots \le t_{r-1}\le 1$,
\begin{equation}\label{i6}
\begin{aligned}
& K_{2,n}(t_1,\dots,t_{r-1})\\
& \qquad \le \frac{c}{n^{q/2+1}} \cdot \Bigl(\EE\Bigl[\prod_{i=1}^{r-1} Y_{n,t_i}\cdot \ind_{A_{n,t_i}}\Bigr] +\EE\Bigl[\Bigr(\prod_{i=1}^{r-1} Y_{n,t_i}\cdot \ind_{A_{n,t_i}}\Bigr)\cdot (\eul_{n,\underline{t_{r-1}}_n+1/n}-\xi)^2\Bigr]\Bigr).
\end{aligned}
\end{equation}
Using 
\eqref{i8} to \eqref{i6}
and employing the induction hypothesis we thus conclude that there exist $c_1,c_2\in (0,\infty)$ such that for all $n\in\N$, 
\begin{equation}\label{i9}
\begin{aligned}
b_{n,r} & \le c_1\cdot\Bigl(
\frac{a_{n,r-1}^{(r-2)/(r-1)}}{n^{q+1}} + \frac{a_{n,r-1}}{n^{q/2+1}}  + \frac{b_{n,r-1}}{n^{q/2+1}} \Bigr)\\
& \le \frac{c_2}{n^{(q+1)r/2}} + \frac{c_1}{n^{q/2+1}} \cdot b_{n,r-1} \le \frac{c_2}{n^{(q+1)r/2}} + \frac{c_1}{n^{(q+1)/2}} \cdot b_{n,r-1}.
 \end{aligned}
\end{equation}

Using~\eqref{i5c} and~\eqref{i9} we obtain by induction  that there exist $c_1,c_2\in (0,\infty)$ such that  for all $n\in \N$,
\begin{equation}\label{i10}
\begin{aligned}
b_{n,r} & \le \frac{c_1}{n^{(q+1)r/2}} + \frac{c_1}{n^{(q+1)(r-1)/2}}\cdot b_{n,1} \le \frac{c_2}{n^{(q+1)r/2}}.
 \end{aligned}
\end{equation}
Inserting the estimate~\eqref{i10} into~\eqref{i5} yields that there 
exists $c\in (0,\infty)$ 
such that  for all $n\in \N$,
\begin{equation}\label{i11}
\begin{aligned}
a_{n,r+1}  
\le \frac{c}{n^{(q+1)(r+1)/2}},
 \end{aligned}
\end{equation}
which completes the proof of~\eqref{prop01}.

We turn to the proof of~\eqref{l33a}. By the definition of $\eul_n$ and by~\eqref{LG} and~\eqref{bound} we see that there exists $c\in (0,\infty)$ such that for all $n\in \N $ and all $t\in [0,1]$,
\begin{equation}\label{v1}
|\eul_{n,t}- \eul_{n,\utn}| \le c\cdot (1+|\eul_{n,\utn}|)\cdot (1/n + |W_t-W_{\utn}| + |W_t-W_{\utn}|^2).
\end{equation} 
Using the fact that for all $n\in\N$, all $t\in [0,1]$ and all $\omega\in A_{n,t}$ we have
\[
|\eul_{n,\utn}(\omega)| \le |\xi| +  |\eul_{n,\utn}(\omega)-\xi| \le |\xi| +  |\eul_{n,t}(\omega) -\eul_{n,\utn}(\omega)|
\]
 we therefore conclude that
 there exists $c\in (0,\infty)$ such that for all $n\in \N $ and all $t\in [0,1]$,
\begin{equation}\label{z232}
|\eul_{n,t}- \eul_{n,\utn}|\cdot \ind_{A_{n,t}} \le c\cdot ( |W_t-W_{\utn}|\cdot \ind_{A_{n,t}} + R_{n,t}),
\end{equation}
where
\begin{align*}
R_{n,t} & = (1+|\eul_{n,t}-\eul_{n,\utn}|)\cdot (1/n + |W_t-W_{\utn}|^2) + |\eul_{n,t}-\eul_{n,\utn}|\cdot |W_t-W_{\utn}|.
\end{align*}
Employing Lemma~\ref{eulprop} we obtain that 
for every $r\in \N$ 
there exists $c\in (0,\infty)$ such that for all $n\in \N $ and all $t\in [0,1]$,
\begin{equation}\label{zt2}
\EE\bigl[|R_{n,t}|^r\bigr] \le c/n^r,
\end{equation}
which yields that there exists $c\in (0,\infty)$ such that for all $n\in \N$,
\begin{equation}\label{zt3}
\EE\Bigl[\Bigl|\int_0^1 |R_{n,t}|^q\,dt\Bigr|^{p}\Bigr]^{1/p} \le c/n^{q}.
\end{equation}
Using~\eqref{z232} and~\eqref{zt3} as well as~\eqref{prop01} we conclude that there exist 
$c_1, c_2\in (0,\infty)$
such that for all $n\in \N$,
\begin{equation*}
\begin{aligned}
& \EE\Bigl[\Bigl|\int_0^1  |\eul_{n,t}-\eul_{n,\utn}|^q\cdot \ind_{A_{n,t}}\,dt\Bigr|^p\Bigr]^{1/p} \\
& \qquad\qquad \le c_1\cdot \EE\Bigl[\Bigl|\int_0^1  |W_{t}-W_{\utn}|^q\cdot \ind_{A_{n,t}}\,dt\Bigr|^p\Bigr]^{1/p} + c_1\cdot \EE\Bigl[\Bigl|\int_0^1 |R_{n,t}|^q\,dt\Bigr|^p\Bigr]^{1/p} \\
& \qquad\qquad \le c_2\cdot (1/n^{(q+1)/2} + 1/n^q) \le 2c_2/n^{(q+1)/2},
\end{aligned}
\end{equation*}
which finishes the proof of the proposition.
\end{proof}

\subsection{Proof of the estimates~\eqref{ll3} and~\eqref{ll33}.}\label{4.3} 

For $n\in\N$ and $t\in [0,1]$ we put 
\[
A_t = \int_0^t \mu(X_s)\, ds,\quad \widehat A_{n,t}  = \int_0^t  \mu(\eul_{n,\usn})\, ds 
\]
and
\[
B_t = \int_0^t \sigma(X_s)\, dW_s,\quad \widehat B_{n,t} = \int_0^t \bigl(\sigma(\eul_{n,\usn}) + \sigma\delta_\sigma(\eul_{n,\usn})\cdot (W_s-W_{\usn})\bigr)\, dW_s
\]
as well as
\[
U_{n,t} = \int_0^t \sigma\delta_\mu(\eul_{n,\usn})\cdot (W_s-W_{\usn})\, ds
\]
and we use the decomposition
\begin{equation}\label{end1}
X_t -\eul_{n,t} = (A_t-\widehat A_{n,t}-U_{n,t}) +  (B_t-\widehat B_{n,t}) + U_{n,t}.
\end{equation}

Furthermore, we put
\[
S_\mu = \Bigl(\bigcup_{\ell=1}^{k_\mu+1} (\xi_{\ell-1},\xi_\ell)^2\Bigr)^c,\quad S_\sigma = \Bigl(\bigcup_{\ell=1}^{k_\sigma+1} (\eta_{\ell-1},\eta_\ell)^2\Bigr)^c
\]
and we note that $S_\mu=\cup_{\ell=1}^{k_\mu}\{(x,y)\in \R^2\colon (x-\xi_\ell)\cdot (y-\xi_\ell)\le 0\}$ and $S_\sigma=\cup_{\ell=1}^{k_\sigma}\{(x,y)\in \R^2\colon (x-\eta_\ell)\cdot (y-\eta_\ell)\le 0\}$. Observing the assumption (B3) we thus obtain by Proposition~\ref{prop1} that there exists $c\in (0,\infty)$ such that for all $n\in\N$ and $q\in \{1,2\}$,  
\begin{equation}\label{end2}
\EE\Bigl[\Bigl|\int_0^1 |\eul_{n,t}-\eul_{n,\utn}|^{q}\cdot \ind_{S_\mu \cup S_\sigma} (\eul_{n,t},\eul_{n,\utn})\, dt\Bigr|^p\Bigr] \le c/n^{p(q+1)/2}.
\end{equation}

For all $n\in\N$ and  $t\in [0,1]$ we have
\begin{equation}\label{end3}
\begin{aligned}
& |\mu(X_t) - \mu(\eul_{n,\utn}) - \sigma \delta_\mu(\eul_{n,\utn})\cdot (W_t-W_{\utn})| \\  & \qquad \le |\mu(X_t) - \mu(\eul_{n,t})| + |\mu(\eul_{n,t})-\mu(\eul_{n,\utn})-\delta_\mu(\eul_{n,\utn})\cdot(\eul_{n,t}-\eul_{n,\utn})|\\  & \qquad \qquad + |\delta_\mu(\eul_{n,\utn})\cdot ( \eul_{n,t}-\eul_{n,\utn}- \sigma(\eul_{n,\utn})\cdot (W_t-W_{\utn}))|\\
& \qquad  = |\mu(X_t) - \mu(\eul_{n,t})|\\
& \qquad \qquad + |\mu(\eul_{n,t})-\mu(\eul_{n,\utn})-\delta_\mu(\eul_{n,\utn})\cdot(\eul_{n,t}-\eul_{n,\utn})|\cdot \ind_{S_\mu^c}(\eul_{n,t},\eul_{n,\utn})\\
& \qquad \qquad + |\mu(\eul_{n,t})-\mu(\eul_{n,\utn})-\delta_\mu(\eul_{n,\utn})\cdot(\eul_{n,t}-\eul_{n,\utn})|\cdot \ind_{S_\mu}(\eul_{n,t},\eul_{n,\utn})\\
 & \qquad \qquad + \bigl|\delta_\mu(\eul_{n,\utn})\cdot ( \mu(\eul_{n,\utn})(t-\utn) + \tfrac{1}{2}\sigma\delta_\sigma(\eul_{n,\utn})\cdot ((W_t-W_{\utn})^2 -(t-\utn)))\bigr|.
\end{aligned}
\end{equation}
Using the assumption (B1) as well as~\eqref{LG},~\eqref{bound} and~\eqref{taylor} we thus obtain that there exists $c\in (0,\infty)$ such that for all $n\in\N$ and all $t\in[0,1]$,
\begin{equation}\label{end4}
\begin{aligned}
& |\mu(X_t) - \mu(\eul_{n,\utn}) - \sigma \delta_\mu(\eul_{n,\utn})\cdot (W_t-W_{\utn})| \\ 
& \qquad \le c\cdot |X_t - \eul_{n,t}| +  c\cdot |\eul_{n,t} - \eul_{n,\utn}|^2 + c\cdot |\eul_{n,t} - \eul_{n,\utn}|\cdot \ind_{S_\mu}(\eul_{n,t},\eul_{n,\utn})\\
& \qquad \qquad + c\cdot (1+|\eul_{n,\utn}|)\cdot (1/n + |W_t-W_{\utn}|^2).
\end{aligned}
\end{equation}
Using~\eqref{end4} as well as Lemma~\ref{eulprop} and~\eqref{end2} with $q=1$ we conclude that there exist $c_1,c_2\in (0,\infty)$ such that for all $n\in\N$ and all $t\in [0,1]$,
\begin{equation}\label{end5}
\begin{aligned}
& \EE\Bigl[\,\sup_{0\le s\le t}|A_s-\widehat A_{n,s}-U_{n,s}|^p\Bigr] \\
& \qquad \qquad\le \EE\Bigl[\int_0^t |\mu(X_s) - \mu(\eul_{n,\usn}) - \sigma \delta_\mu(\eul_{n,\usn})\cdot (W_s-W_{\usn})|^p\, ds\Bigr]\\
& \qquad\qquad \le c_1\cdot \int_0^t \EE\bigl[|X_s - \eul_{n,s}|^p\bigr]\, ds + c_1\cdot \int_0^t \EE\bigl[|\eul_{n,s} - \eul_{n,\usn}|^{2p}\bigr]\, ds\\
& \qquad \qquad \qquad+ c_1\cdot \EE\Bigl[\Bigl|\int_0^t |\eul_{n,s}-\eul_{n,\usn}|\cdot \ind_{S_\mu} (\eul_{n,s},\eul_{n,\usn})\, ds\Bigr|^p\Bigr] \\
& \qquad \qquad \qquad+ c_1\cdot \int_0^t \EE\bigl[(1+|\eul_{n,\usn}|^p)\cdot (1/n^p + |W_s-W_{\usn}|^{2p}\bigr]\, ds\\
& \qquad\qquad \le c_1\cdot \int_0^t \EE[|X_s - \eul_{n,s}|^p]\, ds + c_2/ n^p.
\end{aligned}
\end{equation}

Proceeding similarly to~\eqref{end3} and~\eqref{end4} one obtains that there exists $c\in (0,\infty)$ such that for all $n\in\N$ and all $t\in [0,1]$,
\begin{equation}\label{end6}
\begin{aligned}
& |\sigma(X_t) - \sigma(\eul_{n,\utn}) - \sigma \delta_\sigma(\eul_{n,\utn})\cdot (W_t-W_{\utn})| \\ 
& \qquad \le c\cdot |X_t - \eul_{n,t}| +  c\cdot |\eul_{n,t} - \eul_{n,\utn}|^2 + c\cdot |\eul_{n,t} - \eul_{n,\utn}|\cdot \ind_{S_\sigma}(\eul_{n,t},\eul_{n,\utn})\\
& \qquad \qquad + c\cdot (1+|\eul_{n,\utn}|)\cdot (1/n + |W_t-W_{\utn}|^2).
\end{aligned}
\end{equation}
Employing the Burkholder-Davis-Gundy inequality,  Lemma~\ref{eulprop},~\eqref{end2} with $q=2$ and~\eqref{end6} we then conclude analogously to the derivation of~\eqref{end5} that there exist $c_1,c_2,c_3\in (0,\infty)$ such that for all $n\in\N$ and all $t\in [0,1]$,
\begin{equation}\label{end7}
\begin{aligned}
& \EE\Bigl[\,\sup_{0\le s\le t}|B_t-\widehat B_{n,t}|^p\Bigr] \\
& \qquad \le c_1\cdot \EE\Bigl[\Bigl|\int_0^t |\sigma(X_s) - \sigma(\eul_{n,\usn}) - \sigma \delta_\sigma(\eul_{n,\usn})\cdot (W_s-W_{\usn})|^2\, ds\Bigr|^{p/2}\Bigr] \\
& \qquad\qquad \le c_2\cdot \int_0^t \EE\bigl[|X_s - \eul_{n,s}|^p\bigr]\, ds + c_2\cdot \int_0^t \EE\bigl[|\eul_{n,s} - \eul_{n,\usn}|^{2p}\bigr]\, ds\\
& \qquad \qquad \qquad+ c_2\cdot \EE\Bigl[\Bigl|\int_0^t |\eul_{n,s}-\eul_{n,\usn}|^2\cdot \ind_{S_\sigma} (\eul_{n,s},\eul_{n,\usn})\, ds\Bigr|^{p/2}\Bigr] \\
& \qquad \qquad \qquad+ c_2\cdot \int_0^t \EE\bigl[(1+|\eul_{n,\usn}|^p)\cdot (1/n^p + |W_s-W_{\usn}|^{2p}\bigr]\, ds\\
& \qquad\qquad \le c_2\cdot \int_0^t \EE[|X_s - \eul_{n,s}|^p]\, ds + c_3/ n^{3p/4}.
\end{aligned}
\end{equation}

Combining~\eqref{end1} with~\eqref{end5} and~\eqref{end7} we see that there exists $c\in (0,\infty)$ such that for 
all $n\in\N$ and
all $t\in[0,1]$,
\begin{equation}\label{end8}
\begin{aligned}
\EE\bigl[\,\sup_{0\le s\le t}|X_t-\eul_{n,t}|^p\bigr] \le c\cdot \int_0^t \EE\bigl[\sup_{0\le u\le s}|X_u - \eul_{n,u}|^p\bigr]\, ds + c/n^{3p/4} + \EE\bigl[\,\sup_{0\le s\le t}|U_{n,s}|^p\bigr].
\end{aligned}
\end{equation}
Note that $\EE\bigl[\|X-\eul_{n}\|_\infty^p\bigr] < \infty$ due to~\eqref{mom} and Lemma~\ref{eulprop}.
Below we show that there exists $c\in (0,\infty)$ such that for all $n\in\N$,
\begin{equation}\label{end9}
\begin{aligned}
\EE\bigl[\,\sup_{0\le s\le 1}|U_{n,s}|^p\bigr] \le c/n^p.
\end{aligned}
\end{equation}
Inserting~\eqref{end9} into~\eqref{end8} and applying the Gronwall inequality then yields the error estimate in Theorem~\ref{Thm1}.

We turn to the proof of~\eqref{end9}. Clearly, for all $n\in\N$, all $\ell\in\{0,\dots,n-1\}$ and all $s\in [\ell/n,(\ell+1)/n]$ we have
\begin{equation}\label{end10}
U_{n,s} = U_{n,\ell/n} + \sigma \delta_\mu(\eul_{n,\ell/n})\cdot \int_{\ell/n}^{s}(W_u-W_{\ell/n})\, du,
\end{equation}
which jointly with Lemma~\ref{eulprop} shows that the sequence $(U_{n,\ell/n})_{\ell=0,\dots,n}$ is a martingale. 
Furthermore, using~\eqref{LG} and~\eqref{bound}  we obtain from~\eqref{end10} that there exists $c\in (0,\infty)$ such that for all $n\in\N$,
\begin{equation}\label{end11}
\begin{aligned}
\sup_{0\le s\le 1} |U_{n,s}| & \le \max_{\ell=0,\dots,n-1}|U_{n,\ell/n}| +\max_{\ell=0,\dots,n-1}|\sigma \delta_\mu(\eul_{n,\ell/n})|\cdot \int_{\ell/n}^{(\ell+1)/n}|W_u-W_{\ell/n}|\, du\\
& \le \max_{\ell=0,\dots,n}|U_{n,\ell/n}| + c\cdot (1+\|\eul_n\|_\infty)\cdot \max_{\ell=0,\dots,n-1}\int_{\ell/n}^{(\ell+1)/n}|W_u-W_{\ell/n}|\, du.
\end{aligned}
\end{equation}
Clearly, for all $q\in [1,\infty)$ there exists $c\in (0,\infty)$ such that for all $n\in\N$,
\begin{equation}\label{end12}
\EE\Bigl[\Bigl|\int_{\ell/n}^{(\ell+1)/n}|W_u-W_{\ell/n}|\, du\Bigr|^q\Bigr] \le  \EE\Bigl[\frac{1}{n^{q-1}}\cdot \int_{\ell/n}^{(\ell+1)/n}|W_u-W_{\ell/n}|^{q}\, du\Bigr]\le  \frac{c}{n^{3q/2}}. 
\end{equation}
Employing the Burkholder-Davis-Gundy inequality as well as~\eqref{LG},~\eqref{bound}, Lemma~\ref{eulprop} and~\eqref{end12} we obtain that there exist $c_1,c_2,c_3\in (0,\infty)$ such that for all $n\in\N$,
\begin{equation}\label{end13}
\begin{aligned}
\EE\bigl[\, \max_{\ell=0,\dots,n}|U_{n,\ell/n}|^p\Bigr] & \le 
\EE\Bigl[\Bigl( \sum_{\ell=0}^{n-1} \Bigl(\sigma\delta_\mu(\eul_{n,\ell/n})\cdot \int_{\ell/n}^{(\ell+1)/n}(W_u-W_{\ell/n})\, du\Bigr)^2\Bigr)^{p/2}\Bigr]\\
& \le c_1\cdot \EE\bigl[(1+\|\eul_n\|_\infty^{2p})\bigr]^{1/2}\cdot \EE\Bigl[\Bigl( \sum_{\ell=0}^{n-1} \Bigl( \int_{\ell/n}^{(\ell+1)/n}(W_u-W_{\ell/n})\, du\Bigr)^2\Bigr)^{p}\Bigr]^{1/2}\\
& \le c_2\cdot \Bigl( \sum_{\ell=0}^{n-1}\EE\Bigl[\Bigl( \int_{\ell/n}^{(\ell+1)/n}|W_u-W_{\ell/n}|\, du\Bigr)^{2p}\Bigr]^{1/p}\Bigr)^{p/2} \le \frac{c_3}{n^p}.
\end{aligned}
\end{equation}
Furthermore, by~\eqref{end12} and Lemma~\ref{eulprop} we see that there exists 
$c_1, c_2\in (0,\infty)$
such that for all $n\in\N$,
\begin{equation}\label{end14}
\begin{aligned}
& \EE\Bigl[\Bigl((1+\|\eul_n\|_\infty)
\cdot \max_{\ell = 0,\ldots,k-1}\int_{\ell/n}^{(\ell+1)/n}|W_u-W_{\ell/n}|\, du\Bigr)^p\Bigr] \\
& \qquad\qquad \le c_1\cdot \EE\bigl[(1+\|\eul_n\|_\infty^{2p})\bigr]^{1/2}
\cdot\EE\Bigl[ \sum_{\ell=0}^{n-1} \Bigl( \int_{\ell/n}^{(\ell+1)/n}(W_u-W_{\ell/n})\, du\Bigr)^{2p}\Bigr]^{1/2}
 \le \frac{c_2}{n^p}.
\end{aligned}
\end{equation}
Combining~\eqref{end11} with~\eqref{end13} and~\eqref{end14} 
yields~\eqref{end9}
and completes the proof of the estimate~\eqref{ll3} in Theorem~\ref{Thm1}.

It remains to prove~\eqref{ll33}. In the case $k_\sigma=0$ we have $S_\sigma=\emptyset$. Then the estimates~\eqref{end5} and~\eqref{end9} still hold true but instead of the estimate~\eqref{end7} we obtain 
\begin{equation}\label{endx1}
 \EE\bigl[\,\sup_{0\le s\le t}|B_t-\widehat B_{n,t}|^p\bigr] \le c\cdot \int_0^t \EE\bigl[|X_s - \eul_{n,s}|^p\bigr]\, ds + c_3/ n^{p},
\end{equation} 
where $c\in (0,\infty)$ neither  depends on $n$ nor on $t$. Combining~\eqref{end5},~\eqref{end9} and~\eqref{endx1} we obtain that there exists $c\in (0,\infty)$ such that for all $n\in\N$ and all $t\in [0,1]$,
\begin{equation}\label{last}
 \EE\bigl[\,\sup_{0\le s\le t}|X_t-\eul_{n,t}|^p\bigr] \le c\cdot \int_0^t \EE\bigl[\sup_{0\le u\le s}|X_u - \eul_{n,u}|^p\bigr]\, ds + c/n^{p}.
\end{equation}
Applying the Gronwall inequality we now obtain the estimate~\eqref{ll33} from~\eqref{last}

\section{Proof of Lemmas~\ref{lemx1},~\ref{transform1}}\label{Lem}

We make use of the following result, which is straightforward to check.

\begin{lemma}\label{mult}
Let $-\infty \le a < b \le \infty$ and let $f, g\colon\R\to\R$ be Lipschitz continuous on  $(a,b)$. Assume further that there exists $c\in (0, \infty)$ such that $g$ is constant on the set $(-\infty,c)\cup (c,\infty)$. Then  $f\cdot g$ is Lipschitz continuous on $(a,b)$.
\end{lemma}

\subsection{Proof of Lemma~\ref{lemx1}}
We first show that $G_{z,\alpha,\nu}$ satisfies (i).
It is  straightforward to check that $G_{z,\alpha,\nu}$ is differentiable on $\R$ with 
\[
G_{z,\alpha,\nu}'(x)=1+\sum_{i=1}^k 2\alpha_i \nu\cdot \frac{|x-z_i|}{\nu}\cdot \Bigl(1-\Bigl(\frac{x-z_i}{\nu}\Bigr)^2\Bigr)^3\cdot \Bigl(1-5\,\Bigl(\frac{x-z_i}{\nu}\Bigr)^2\Bigr)\cdot \ind_{[z_i-\nu, z_i+\nu]}(x)
\]
for all $x\in\R$. Note that for every $i\in\{1,\dots,k\}$ the mapping $x\mapsto \frac{|x-z_i|}{\nu}\cdot (1-(\frac{x-z_i}{\nu})^2)^3\cdot (1-5\,(\frac{x-z_i}{\nu})^2)\cdot \ind_{[z_i-\nu, z_i+\nu]}(x)$ is Lipschitz continuous on $\R$. Thus, as a finite linear combination of Lipschitz continuous functions,  $G_{z,\alpha,\nu}'$ is Lipschitz continuous on $\R$ as well. Clearly, 
for all $x\in \{z_1,\dots,z_k\}\cup\R\setminus \bigcup_{i=1}^k (z_i-\nu, z_i+\nu)$
we have
\begin{equation}\label{first}
G_{z,\alpha,\nu}'(x)=1
\end{equation}
and for all $i\in\{1, \ldots, k\}$ and all $x\in [\xi_i-\nu, \xi_i+\nu]$ we have
\[
G_{z,\alpha,\nu}'(x)\geq 1-8|\alpha_i| \nu>0,
\]
which finishes the proof of part (i) of the lemma.

Next we show that $G_{z,\alpha,\nu}$ satisfies (ii) and (iii). Note that the intervals $[z_i-\nu, z_i+\nu]$, $i=1,\dots,k$, are pairwise disjoint. Observing~\eqref{first} it is easy to check that $G_{z,\alpha,\nu}'$ is two times differentiable on $ \cup_{i=1}^{k+1}(z_{i-1},z_i)$ with
\begin{equation}\label{der}
\begin{aligned}
G_{z,\alpha,\nu}''(x) & = \begin{cases} -2\alpha_i \cdot \psi_i(x), & \text{if }x\in (z_i-\nu,z_i),\\2\alpha_i \cdot \psi_i(x), & \text{if }x\in (z_i, z_i+\nu),\\
0, & \text{if }x\in \R\setminus \bigcup_{j=1}^k (z_j-\nu, z_j+\nu),\end{cases} \\
G_{z,\alpha,\nu}'''(x) & = \begin{cases} -2\alpha_i/\nu \cdot \eta_i(x), & \text{if }x\in (z_i-\nu,z_i),\\
2\alpha_i/\nu \cdot \eta_i(x), & \text{if }x\in (z_i, z_i+\nu),\\
0, & \text{if }x\in \R\setminus \bigcup_{j=1}^k (z_j-\nu, z_j+\nu),\end{cases}
\end{aligned}
\end{equation}
where
\begin{equation}\label{der2}
\begin{aligned}
\psi_i(x) & = 
\Bigl(1-\Bigl(\frac{x-z_i}{\nu}\Bigr)^2\Bigr)^2\Bigl(1-22\Bigl(\frac{x-z_i}{\nu}\Bigr)^2 + 45\Bigl(\frac{x-z_i}{\nu}\Bigr)^4\Bigr),\\
\eta_i(x) & = \Bigl(1-\Bigl(\frac{x-z_i}{\nu}\Bigr)^2\Bigr)\Bigl(-48\frac{x-z_i}{\nu}+ 312\Bigl(\frac{x-z_i}{\nu}\Bigr)^3-360\Bigl(\frac{x-z_i}{\nu}\Bigr)^5 \Bigr).
\end{aligned}
\end{equation}
Obviously, on each interval $(z_{i-1},z_i)$,   $G''_{z,\alpha,\nu}$ and $G'''_{z,\alpha,\nu}$ are Lipschitz continuous, and we have $G''_{z,\alpha,\nu}(z_i-) = -2\alpha_i = -G''_{z,\alpha,\nu}(z_i+)$.
This finishes the proof of Lemma~\ref{lemx1}.\qed

\subsection{Proof of Lemma~\ref{transform1}}

Due to Lemma~\ref{lemx1} the functions $\widetilde \mu$ and $\widetilde \sigma$ are well-defined. Recall from Lemma~\ref{lemx1}(i) that there exists $c\in (0,\infty)$ such that $G'_{\xi,\alpha,\nu}=1$ on $(-\infty,c)\cup (c,\infty)$. Hence $G''_{\xi,\alpha,\nu}=0$ on $(-\infty,c)\cup (c,\infty)$. By means of Lemma~\ref{mult} we can thus conclude that $ G'_{\xi,\alpha,\nu}\cdot\mu$ and $G''_{\xi,\alpha,\nu}\cdot \sigma^2$ are  Lipschitz continuous on each of the intervals  $(\xi_{0},\xi_1),\dots,(\xi_{k},\xi_{k+1})$ and that $ G'_{\xi,\alpha,\nu}\cdot\sigma$ is Lipschitz continuous on $\R$. Observing Lemma~\ref{lemx1}(i),(iii) we see that for each $i\in\{1,\dots,k\}$,
\begin{align*}
(G_{\xi,\alpha,\nu}'\cdot \mu +\tfrac{1}{2}G_{\xi,\alpha,\nu}''\cdot\sigma^2)(\xi_i-) & = \mu(\xi_i-) -\alpha_i\cdot \sigma^2(\xi_i)\\ &
 = (\mu(\xi_i-) +\mu(\xi_i+))/2  = (G_{\xi,\alpha,\nu}'\cdot \mu +\tfrac{1}{2}G_{\xi,\alpha,\nu}''\cdot\sigma^2)(\xi_i)\\
& = \mu(\xi_i+) +\alpha_i\cdot \sigma^2(\xi_i) =  (G_{\xi,\alpha,\nu}'\cdot \mu+\tfrac{1}{2}G_{\xi,\alpha,\nu}''\cdot\sigma^2)(\xi_i+).
\end{align*}
Hence $G_{\xi,\alpha,\nu}'\cdot \mu +\tfrac{1}{2}G_{\xi,\alpha,\nu}''\cdot\sigma^2$ is continuous on $\R$ and Lipschitz continuous on each of the intervals  $(\xi_{0},\xi_1),\dots,(\xi_{k},\xi_{k+1})$, which yields Lipschitz continuity of the latter function on the whole real line. Finally, recall  that by Lemma~\ref{lemx1},   $G^{-1}_{\xi,\alpha,\nu}$ is Lipschitz continous. This shows that $\widetilde \mu$ and $\widetilde \sigma$ satisfy 
the assumption
(B1).

Using 
the assumption
(A3), Lemma~\ref{lemx1}(i),(ii) 
and the fact that $G^{-1}((\xi_{i-1},\xi_i))=(\xi_{i-1},\xi_i)$ for all $i\in\{1, \ldots, k+1\}$
we immediately obtain  that for each $i\in\{1,\dots,k+1\}$ the functions $\widetilde \mu$ and $\widetilde \sigma$ are differentiable on $(\xi_{i-1},\xi_i)$ with derivatives
\begin{align*}
\widetilde \mu' & = (\mu' + G_{\xi,\alpha,\nu}''/G_{\xi,\alpha,\nu}'\cdot (\mu + \sigma\cdot \sigma') + \tfrac{1}{2}G_{\xi,\alpha,\nu}'''/G_{\xi,\alpha,\nu}'\cdot \sigma^2)\circ G^{-1}_{\xi,\alpha,\nu},\\
\widetilde \sigma' & = ( \sigma' + G_{\xi,\alpha,\nu}''/G_{\xi,\alpha,\nu}'\cdot \sigma) \circ G^{-1}_{\xi,\alpha,\nu}.
\end{align*}
Using
the assumption
(A3) and Lemma~\ref{lemx1}(i),(ii) again we can now derive by iteratively applying Lemma~\ref{mult} (with any extension of $\mu'$ and $\sigma'$ to the whole real line)  that for each $i\in\{1,\dots,k+1\}$ the functions $\widetilde \mu'$ and $\widetilde \sigma'$ are  Lipschitz continous on  $(\xi_{i-1},\xi_i)$. Hence  $\widetilde \mu$ and $\widetilde \sigma$ satisfy 
the assumption
(B2) with $k_\mu=k_\sigma = k $ and $\eta_i = \xi_i$ for $i\in\{1,\dots,k\}$. Finally, note that $G_{\xi,\alpha,\nu}(\xi_i) = \xi_i$ for each $i\in\{1,\dots,k\}$, which yields that $\widetilde \sigma(\xi_i) = \sigma(\xi_i) \neq 0$ for each $i\in\{1,\dots,k\}$. Hence $\widetilde \sigma$ satisfies 
the assumption
(B3), which finishes the proof of Lemma~\ref{transform1}.\qed

\bibliographystyle{acm}
\bibliography{bibfile}

\def\cprime{$'$} \def\cprime{$'$}
\begin{thebibliography}{10}

\bibitem{GLN17}
{\sc G{\"o}ttlich, S., Lux, K., and Neuenkirch, A.}
\newblock The {E}uler scheme for stochastic differential equations with
  discontinuous drift coefficient: {A} numerical study of the convergence rate.
\newblock {\em arXiv:1705.04562\/} (2017), 18 pages.

\bibitem{g98b}
{\sc Gy{\"o}ngy, I.}
\newblock A note on {E}uler's approximations.
\newblock {\em Potential Anal. 8}, 3 (1998), 205--216.

\bibitem{gk96b}
{\sc Gy{\"o}ngy, I., and Krylov, N.}
\newblock Existence of strong solutions for {I}t{\^o}'s stochastic equations
  via approximations.
\newblock {\em Probab. Theory Related Fields 105}, 2 (1996), 143--158.

\bibitem{HalidiasKloeden2008}
{\sc Halidias, N., and Kloeden, P.~E.}
\newblock A note on the {E}uler-{M}aruyama scheme for stochastic differential
  equations with a discontinuous monotone drift coefficient.
\newblock {\em BIT 48}, 1 (2008), 51--59.

\bibitem{hhmg2019}
{\sc Hefter, M., Herzwurm, A., and M{\"u}ller-Gronbach, T.}
\newblock Lower error bounds for strong approximation of scalar sdes with
  non-lipschitzian coefficients.
\newblock {\em Ann. Appl. Probab. 29}, 1 (2019), 178--216.

\bibitem{HMGR01}
{\sc Hofmann, N., M\"{u}ller-Gronbach, T., and Ritter, K.}
\newblock The optimal discretization of stochastic differential equations.
\newblock {\em J. Complexity 17\/} (2001), 117--153.

\bibitem{kp92}
{\sc Kloeden, P.~E., and Platen, E.}
\newblock {\em Numerical solution of stochastic differential equations},
  vol.~23 of {\em Applications of Mathematics (New York)}.
\newblock Springer-Verlag, Berlin, 1992.

\bibitem{KruseWu18}
{\sc Kruse, R., and Wu, Y.}
\newblock {A randomized Milstein method for stochastic differential equations
  with non-differentiable drift coefficients}.
\newblock {\em Discrete Contin. Dyn. Syst. Ser. B (online first)\/} (2018).

\bibitem{LS16}
{\sc Leobacher, G., and Sz{\"o}lgyenyi, M.}
\newblock A numerical method for {SDE}s with discontinuous drift.
\newblock {\em BIT 56}, 1 (2016), 151--162.

\bibitem{LS15b}
{\sc Leobacher, G., and Sz{\"o}lgyenyi, M.}
\newblock A strong order 1/2 method for multidimensional {SDE}s with
  discontinuous drift.
\newblock {\em Ann. Appl. Probab. 27\/} (2017), 2383--2418.

\bibitem{LS18}
{\sc Leobacher, G., and Sz\"olgyenyi, M.}
\newblock Convergence of the {E}uler-{M}aruyama method for multidimensional
  {SDE}s with discontinuous drift and degenerate diffusion coefficient.
\newblock {\em Numer. Math. 138}, 1 (2018), 219--239.

\bibitem{Mao08}
{\sc Mao, X.}
\newblock {\em Stochastic differential equations and applications}, second~ed.
\newblock Horwood Publishing Limited, Chichester, 2008.

\bibitem{m04}
{\sc M{\"u}ller-Gronbach, T.}
\newblock Optimal pointwise approximation of {SDE}s based on {B}rownian motion
  at discrete points.
\newblock {\em Ann. Appl. Probab. 14}, 4 (2004), 1605--1642.

\bibitem{MGY18a}
{\sc M{\"u}ller-Gronbach, T., and Yaroslavtseva, L.}
\newblock On the performance of the {E}uler-{M}aruyama scheme for {SDE}s with
  discontinuous drift coefficient.
\newblock {\em arXiv:1809.08423\/} (2018).

\bibitem{NSS19}
{\sc Neuenkirch, A., Sz\"olgyenyi, M., and Szpruch, L.}
\newblock An adaptive {E}uler-{M}aruyama scheme for stochastic differential
  equations with discontinuous drift and its convergence analysis.
\newblock {\em {SIAM J. Numer. Anal.} 57\/} (2019), 378--403.

\bibitem{Tag16}
{\sc Ngo, H.-L., and Taguchi, D.}
\newblock Strong rate of convergence for the {E}uler-{M}aruyama approximation
  of stochastic differential equations with irregular coefficients.
\newblock {\em Math. Comp. 85}, 300 (2016), 1793--1819.

\bibitem{Tag2017b}
{\sc Ngo, H.-L., and Taguchi, D.}
\newblock On the {E}uler-{M}aruyama approximation for one-dimensional
  stochastic differential equations with irregular coefficients.
\newblock {\em IMA J. Numer. Anal. 37}, 4 (2017), 1864--1883.

\bibitem{Tag2017a}
{\sc Ngo, H.-L., and Taguchi, D.}
\newblock Strong convergence for the {E}uler-{M}aruyama approximation of
  stochastic differential equations with discontinuous coefficients.
\newblock {\em Statist. Probab. Lett. 125\/} (2017), 55--63.

\bibitem{RevuzYor2005}
{\sc Revuz, D., and Yor, M.}
\newblock {\em Continuous martingales and {B}rownian motion}, third~ed.
\newblock Springer-Verlag, Berlin, 1995.

\end{thebibliography}

\end{document}